\documentclass[11pt]{amsart}
\usepackage{mathrsfs,amssymb}
\usepackage{amsfonts}
\usepackage{amssymb}
\usepackage{amsmath}
\usepackage{amsthm}
\usepackage{relsize}
\usepackage[mathscr]{euscript}

\newtheorem{theorem}{Theorem}[section]
\newtheorem{lemma}[theorem]{Lemma}

\newtheorem{corollary}[theorem]{Corollary}
\newtheorem{conjecture}[theorem]{Conjecture}

\newtheorem{Remark}[theorem]{Remark}

\newcommand{\R}{\mathbb{R}}

\newcommand{\Z}{\mathbb{Z}}

\newcommand{\N}{\mathbb{N}}

\newcommand{\p}{\mathfrak{p}}

\newcommand{\A}{\mathfrak{a}}

\begin{document}

\title[Symplectic Family One-Level Density]{Lower Order Terms for the One-Level Density of a Symplectic Family of Hecke L-Functions}
\author{Ezra Waxman}
\date{\today}
\address{Tel Aviv University, School of Mathematical Sciences, Tel-Aviv 69978, Israel}
\address{Charles University, Faculty of Mathematics and Physics, Department of Algebra, Sokolovska 83, 18600 Praha 8, Czech Republic}
\address{Technische Universität Dresden, Fakultät Mathematik, Institut für Algebra, Zellescher Weg 12-14, 01062 Dresden, Germany}
\email{ezrawaxman@gmail.com}
\numberwithin{equation}{section}
\maketitle
\begin{abstract}
In this paper we apply the $L$-function Ratios Conjecture to compute the one-level density for a symplectic family of $L$-functions attached to Hecke characters of infinite order.  When the support of the Fourier transform of the corresponding test function $f$ reaches $1$, we observe a transition in the main term, as well as in the lower order term.  The transition in the lower order term is in line with behavior recently observed by D. Fiorilli, J. Parks, and A. S\"odergren in their study of a symplectic family of quadratic Dirichlet $L$-functions \cite{FiorParkS, FiorParkS2}.  We then directly calculate main and lower order terms for test functions $f$ such that supp($\widehat{f}) \subset [-\alpha,\alpha]$ for some $\alpha <1$, and observe that this unconditional result is in agreement with the prediction provided by the Ratios Conjecture.  As the analytic conductor of these L-functions grow twice as large (on a logarithmic scale) as the cardinality of the family in question, this is the optimal support that can be expected with current methods. Finally, as a corollary, we deduce that, under GRH, at least 75$\%$ of these $L$-functions do not vanish at the central point.
\end{abstract}
\section{Introduction}
\subsection{One-Level Densities and Random Matrices}
The statistical distribution of zeros across a family of $L$-functions near the central point has emerged as a popular object of study.  This is due to its role in a range of arithmetic applications, such as determining the size of the Mordell-Weil groups of elliptic curves \cite{BSD, BSD2}, and bounding the size of the class number of imaginary quadratic fields \cite{CI}.  It is also of purely theoretical interest since such information can be used to determine the symmetry type of the family in question.\\
\\
Let $\rho_{\chi} = 1/2 +i \gamma_{\chi}$ denote a generic non-trivial zero of a given $L$-function $L(s,\chi)$.  Fix $f$ to be an even Schwarz function such that $\widehat{f}$ is compactly supported.  The \textbf{$($scaled$)$ one-level density} is then defined as
\begin{equation*}
D_{1}(\chi;f,R):=\sum_{\rho_{\chi}}f\left(\frac{\log R}{\pi} \cdot \gamma_{\chi}\right),
\end{equation*}
where $R$ is a scaling parameter $($typically related to the analytic conductor of the family, $c_{R})$, intended to ensure that the mean spacing between low-lying zeroes is equal to 1.\\
\\
If $\mathcal{F}(R)$ is a family of $L$-functions with bounded analytic conductor, one moreover defines the \textbf{averaged (scaled) one-level density} over the family to be
\begin{equation*}
D_{1}(\mathcal{F}(R);f):=\frac{1}{|\mathcal{F}(R)|}\sum_{\chi \in \mathcal{F}(R)}D_{1}(\chi;f,R).
\end{equation*}

Much research has been dedicated to studying the one-level density for a variety of families, such as $L$-functions attached to Dirichlet characters, elliptic curves, cuspidal newforms, Maass forms, and many other objects \cite{DuMill, Gul, HughRud, ILS}.  Usually, one can compute the one-level density for a limited class of test functions, namely those for which supp($\widehat{f}) \subseteq [-\alpha,\alpha]$, for a particular $\alpha \in \R$.  In all such cases to date, the computation has agreed with that of a random matrix ensemble in the restricted support considered.\\
\\
The ``natural" upper bound on $\alpha$ is determined by a collection of factors, including the analytic conductor of the $L$-functions, and the sparseness of the family in question.  In recent work, Drappeau, Pratt, and Radziwill \cite{DPR} computed the one-level density of low-lying zeros over the family of primitive Dirichlet characters, and reached an upper bound of $\alpha = 2+50/1093$, providing the first unconditional result in which supp($\widehat{f})$ may be extended beyond $\alpha = 2$.  While there is no satisfactory definition, uniform across families, as to what constitutes a natural upper bound, we note that, at present, there are no known families for which the one-level density has been unconditionally computed beyond $\alpha = 2 \log |\mathcal{F}(R)|/\log |c_R|$.  For the particular family of degree-two $L$-functions that we will be studying in this paper (denoted $\mathcal{F}(K)$), the unconditional result (Theorem \ref{oneleveldensity}) is restricted to test functions for which $\alpha < 1$.  The analytic conductor $c_{k} \asymp k^{2}$ of the $L$-functions in question grows twice as large (on a logarithmic scale) as the cardinality of $\mathcal{F}(K)$, and so we consider this to be the natural upper bound for this particular family.
\\
\\
Let $G$ denote an $M \times M$ classical matrix group (i.e. the group of unitary, symplectic, orthogonal, or special orthogonal matrices), and set $dA$ to be the normalized Haar measure on $G$.  As $A \in G$ is unitary, its eigenvalues $\{e^{i\varphi_{1}},e^{i\varphi_{2}},\dots, e^{i\varphi_{M}}\}$ lie on the unit circle, and we may define the ordered sequence

\begin{equation*}
0 \leq \varphi_{1}(A) \leq \varphi_{2}(A) \leq \dots \leq \varphi_{M}(A) < 2\pi.
\end{equation*}
Writing $\varphi_{j+l\cdot M}(A) := \varphi_{j}(A)+2\pi\cdot l$, we may then extend our definition to the set $\{\varphi_{n}\}_{n \in \Z}$ of ``all" real angles of $A$.  Defining $\vartheta (i) := M\varphi_{i}(A)/2\pi$, we then consider the \textit{normalized} sequence $\{\vartheta_{i}\}_{i \in \Z}$ whose average spacing is equal to one.\\
\\
For $f$ as above, define
\begin{equation*}
W_{1}(f,A):=\sum_{n \in \Z}f(\vartheta (n)),
\end{equation*}
as well as 
\begin{equation*}
W_{1}(f,G) := \int_{G} W_{1}(f,A)dA.
\end{equation*}

The \textit{Katz-Sarnak Density Conjecture} \cite{KatzSarn1,KatzSarn2}  states that in the limit as the conductor tends to infinity, $D_{1}(\mathcal{F}(R),f)$ is modeled by the the one-level scaling density of eigenvalues near 1 of the classical compact group $G$ corresponding to the symmetry type of $\mathcal{F}(R)$.  In other words, for an admissible test function $f$, we expect

\begin{equation*}
\lim_{R \rightarrow \infty} D_{1}(\mathcal{F}(R),f) = \lim_{N \rightarrow \infty}W_{1}(f,G)= \int_{\R}f(x)W_{1,G}(x)dx,
\end{equation*}
where
\begin{align}\label{matrix calculation}
W_{1,G} = \left\{
\begin{array}{l l l l l}
1 & \text{ if } G = U \\
1-\frac{\sin 2\pi t}{2\pi t} & \text{ if } G = USp \\
1+\frac{1}{2}\delta_{0}(t) & \text{ if } G = O \\
1+\frac{\sin 2\pi t}{2\pi t} & \text{ if } G = SO \text{(even)} \\
1+\frac{1}{2}\delta_{0}(t)-\frac{\sin 2\pi t}{2\pi t} & \text{ if } G = SO \text{(odd)}
\end{array} \right.
\end{align}
is computed explicitly in \cite[\S AD.12.6]{KatzSarn1}.  Here $\delta_{0}$ is the Dirac distribution, and $U$, $USp$, $O$, $SO\text{(even)}$, $SO\text{(odd)}$ denote the unitary, symplectic, orthogonal, special even-orthogonal, and special odd-orthogonal matrix groups, respectively.\\
\\
Random Matrix Theory has proved extremely successful at modelling a variety of statistical quantities related to the zeta function, and across families of $L$-functions.  Montgomery's work on the pair correlation of the zeta zeros \cite{Montgomery} pointed to a deep connection between statistics attached to zeros of the zeta function and eigenvalues of random matrices, while Rudnick and Sarnak's work showed that the local spacing distribution between zeros high up on the critical strip of any fixed automorphic $L$-function obeys universal behavior in line with Dyson's computation for the GUE model \cite{RS}.  The Katz-Sarnak Density Conjecture extends this connection to a relationship between the low-lying zeros across a family of $L$-functions near the central point and distribution laws for eigenvalues of random matrices.  Here universality breaks down, and, as above, one may use the resulting computation to distinguish between families of different symmetry type.\\
\\
Motivated by calculations for the characteristic polynomials of matrices averaged over the compact classical groups, Conrey, Farmer, and Zirnbauer \cite{Conrey, ConreyII} suggested a recipe for calculating averages of quotients of products of shifted $L$-functions evaluated at certain values on the critical strip.  The \textit{$L$-function Ratios Conjecture} can be used to give precise predictions for a variety of statistics across families of $L$-functions, including the one-level density \cite{ConreySnaith}.  Our goal in this work will be to use the Ratios Conjecture to motivate a precise conjecture for the one-level density of a family of $L$-functions attached to Hecke characters of infinite order, and then to demonstrate the accuracy of this conjecture for a restricted class of test functions.
\subsection{Main Results}
Let $\alpha \in \Z[i]$ be an element in the ring of Gaussian integers and let $\A = \langle \alpha \rangle \subseteq \Z[i]$ denote the ideal generated by $\alpha$.  For any $k \in \Z$, the \textit{Hecke character} $\Xi_{k}(\A):= \left(\alpha/ \overline{\alpha}\right)^{2k}$ provides a well-defined function on the ideals of $\Z[i]$.  Attached to any such character $\Xi_{k}$ is a \textit{Hecke $L$-function} of the form
\begin{align*}
L_{k}(s) &:= \ \sum_{\substack{\A \subseteq \Z[i]\\ \A \neq 0}}\frac{\Xi_{k}(\A)}{N(\A)^{s}}=\prod_{\p \textnormal{ prime}}\left(1-\frac{\Xi_{k}(\p)}{N(\p)^s}\right)^{-1}, \hspace{5mm} \textnormal{Re}(s)>1,
\end{align*}

where $N(\A) := \alpha\overline{\alpha}$ is the \textit{norm} of $\A$, and where $\p$ runs through the prime ideals in $\Z[i]$.  In this work, we study the average one-level density across the family
\begin{equation*}
\mathcal{F}(K) := \{L_{k}(s):1 \leq k \leq K\}.
\end{equation*}

This family of characters arises in several applications connected to the angular distribution of Gaussian primes \cite{Kubilius 1950, Kubilius 1955}, and, in particular, the zero distribution near the central point across $\mathcal{F}(K)$ affects the variance of Gaussian primes across sectors \cite{SMALL paper, RudWax}.\\
\\
Let $\rho_{k} = 1/2+i \gamma_{k}$ denote a generic non-trivial zero of
$L_{k}(s)$.  By Theorem 5.8 in \cite{IwanKow}, the growth of the number of nontrivial zeros of $L(s,\Xi_k)$ in a fixed rectangle is 
\begin{equation}\label{density of zeros}
\#\{\rho_{k}: 0\leq \Im(\rho_{k})\leq T_0\}  \sim \frac {T_0 \log |k|}{\pi}, \quad k\to \infty, \quad T_0>0\;{\rm fixed,}
\end{equation}
so that the density of zeros in $L_{k}(s)$ near the central point is $\log |k|/\pi$.  We thus choose $K$ as our scaling parameter so that the mean value spacing between zeros is 1.  The average one-level density across $\mathcal{F}(K)$ is then
\begin{align*}
D_{1}(\mathcal{F}(K);f)=\frac{1}{K}\sum_{k=1}^{K}\sum_{\gamma_{k}}f\left(\gamma_{k}\frac{ M}{\pi}\right),
\end{align*}
where $M:= \log K$.\\
\\
Set the Fourier transform to be normalized as $f(u) := \int_{\R}\widehat{f}(r)e^{2 \pi i r u}dr.$  Based on a Ratios Conjecture model, we suggest the following conjecture:

\begin{conjecture}\label{ratios prediction}
\begin{align*}
D_{1}(\mathcal{F}(K);f)& = \widehat{f}(0)- \frac{1}{2}\int_{\R}\widehat{f}(x) \cdot 1_{[-1,1]}(x)dx+\frac{1}{M}\bigg(c\cdot \widehat{f}(0)-d\cdot \widehat{f}(1)\bigg)(1+o(1)),
\end{align*}
where
\begin{align}\label{d constant}
d &:= 3 \log 2-1-\textnormal{log }\pi+4\log |\eta(i)|-2\sum_{p\equiv 3(4)}\frac{\log p}{p^2-1},
\end{align}
and 
\begin{align}\label{c constant}
c &:= d -c_{1}-\gamma_{0},
\end{align}
where $c_{1}$ is defined as in \textnormal{(\ref{c1})}.  Here $\eta(s)$ is the \textit{Dedekind eta function}, and
\begin{equation*}
\eta(i)=e^{-\frac{\pi}{12}}\prod_{n=1}^\infty \bigg(1-e^{- 2\pi n}\bigg) = \frac{\Gamma\left(\frac 1 4\right)}{2\pi^{3/4}}.
\end{equation*}
\end{conjecture}

Note that by Plancherel's identity,

\begin{align}\label{small support symplectic}
\int_{\R}f(t)\frac{\sin 2\pi t}{2\pi t}dt=\frac{1}{2}\int_{-1}^{1}\widehat{f}(x) \cdot 1_{[-1,1]}dx.
\end{align}

It then follows from (\ref{matrix calculation}) that Conjecture \ref{ratios prediction} agrees with the Density Conjecture and that $\mathcal{F}(K)$ has \textit{symplectic monodromy}.  This agrees with the findings in the function field setting, where the analogue to $\mathcal{F}(K)$ (the family of so-called \textit{super-even characters}) is also found to have symplectic monodromy \cite{Katz supereven}.  Note, moreover, the emergence of a sharp transition in both the main term and the lower order term of Conjecture \ref{ratios prediction}.  A similar transition was observed for a symplectic family of Dirichlet $L$-functions by D. Fiorilli, J. Parks, and A. S\"odergren \cite{FiorParkS, FiorParkS2}.\\
\\
To arrive at Conjecture \ref{ratios prediction} we make use of Cauchy's residue theorem to write $D_{1}(\mathcal{F}(K);f)$ in terms of two vertical contour integrals, which we then compute by applying the ratios recipe.  The resulting expression involves several terms (depending on $K$ and $f$), which we then explicitly compute: $i$) a term $W_{f}$ emerging from the infinite place of each $L_{k}(s)$ $ii)$ a collection of terms $S_{\zeta},S_{L}$, and $S_{A'}$ emerging from the first piece of the approximate functional equation within the ratios recipe, and $iii)$ a final term $S_{\Gamma}$ emerging from the second piece of the approximate function equation.  When supp$(\widehat{f}) \subset (-1,1)$, $W_{f}$ and $S_{\zeta}$ contribute to the main term, while lower order contributions come from $W_{f}, S_{\zeta},S_{L}$, and $S_{A'}$.  $S_{\Gamma}$ accounts for the sharp transition in both the main and lower order terms, contributing only when supp$(\widehat{f}) \not \subset (-1,1)$.\\
\\
Moreover, we prove the following theorem:

\begin{theorem}\label{oneleveldensity}
Suppose that supp$(\widehat{f}) \subset [-\alpha,\alpha]$ for some $\alpha <1$.  Then
\begin{align}
D_{1}(\mathcal{F}(K);f)&= \widehat{f}(0)- \frac{f(0)}{2}+c\cdot \frac{\widehat{f}(0)}{M}+O\left(\frac{1}{M^{2}}\right),
\end{align}
where $c$ is as above.
\end{theorem}

Note that Theorem \ref{oneleveldensity} agrees with Conjecture \ref{ratios prediction}.  In fact, one may show that if supp$(\widehat{f}) \subset [-\alpha,\alpha]$ for some $\alpha< 1$, then $D_{1}(\mathcal{F}(K);f)$ agrees with the prediction of the Ratios Conjecture down to an accuracy of size $O(M^{-n})$, for any $n \geq 2$ (Remark \ref{remark unconditional match}).  For conciseness, we only explicitly compute the first such lower order term.\\
\\
To prove Theorem \ref{oneleveldensity}, we use the explicit formula to convert the sum over zeros of $L_{k}(s)$ into a sum over primes.  The resulting expression involves a collection of terms: $i)$ $W_{f}$ (as above) emerging from the infinite place of each $L_{k}(s)$ $ii)$ a term $S_{\text{inert}}$ emerging from the inert primes $iii)$ a term $S_{\text{ram}}$ emerging from the ramified prime, and $iv)$ a term $S_{\text{split}}$ emerging from the split primes.  We then show that

\begin{equation}\label{GRH Ratios Relation}
S_{\text{inert}}+S_{\text{ram}} = S_{\zeta}+S_{L}+S_{A'}+O\left(\frac{1}{K}\right),
\end{equation}
and apply our previous explicit computation of the right side of this equation.  When supp$(\widehat{f}) \subset [-\alpha,\alpha]$ for some $\alpha < 1$, we then demonstrate that the contribution from $S_{\text{split}}$ is negligible.  The breadth of the support of the test function in Theorem \ref{oneleveldensity} is thus limited by our inability to explicitly compute $S_{\text{split}}$ unconditionally, when $\alpha \geq 1$.  We may do so, however, conditionally on the Ratios Conjecture; in which case we find that $S_{\text{split}} = S_{\Gamma}+O\left(e^{-\frac{M}{2}}\right)$ for any admissible $f$ (Corollary \ref{character sum}).\\
\\
From Theorem \ref{oneleveldensity} we moreover obtain the following corollary concerning the non-vanishing of $L_{k}(1/2)$.
\begin{corollary}\label{nonvanish}
Assume the \textit{Generalized Riemann Hypothesis} (GRH).  Then
\begin{equation*}
\lim_{K \rightarrow \infty}\frac{1}{|K|}\{L_{k}(s) \in \mathcal{F}(K) \}:L_{k}(1/2) \neq 0 \}\geq \frac{3}{4}+o(1).
\end{equation*}
\end{corollary}
In Section 2 we discuss Hecke characters $\Xi_{k}(\A)$ and their associated $L$-functions.  In Section 3 we compute the Ratios Conjecture recipe for the family $\mathcal{F}(K)$, and then in Section 4 we apply the result to obtain a conjecture for $D_{1}(\mathcal{F}(K);f)$ with a power-savings error term.  This conjecture is expressed via the collection of terms discussed above, which are then explicitly computed in Section 5, resulting in Conjecture \ref{ratios prediction}.  Finally, Section 6 provides a proof of Theorem \ref{oneleveldensity} and of Corollary \ref{nonvanish}.\\
\\
\textbf{Acknowledgments}: We thank Chantal David, Alexei Entin, Daniel Fiorilli, Bingrong Huang, Jared D. Lichtman, Ian Petrow, Anders Södergren, and Zeev Rudnick for helpful discussions.  We also thank the anonymous reviewer, for a careful read-through.  The author was supported by the European Research Council under the European Union's Seventh Framework Programme (FP7/2007-2013) / ERC grant agreement no 320755., as well as by the Czech Science Foundation GA\v CR, grant 17-04703Y, and by a Minerva Post-Doctoral Fellowship.
\section{Hecke Characters}
The notation in this section draws from the general set-ups provided in \cite{RS} and \cite[Ch. 5]{IwanKow}.  Let $\alpha \in \Z[i]$ be an element in the ring of Gaussian integers and let $\theta_{\alpha}$ denote the argument of $\alpha$.  Consider the ideal $\A = \langle \alpha \rangle \subseteq \Z[i]$ generated by $\alpha$.  Since $\Z[i]$ is a principal ideal domain, and the generators of $\A$ differ by multiplication by a unit $\{\pm 1, \pm i\} \in \Z[i]^{\times}$, we find that $\theta_{\A} := \theta_{\alpha}$ is well-defined modulo $\pi/2$.  We may thus fix $\theta_{\A}$ to lie in $[0,\pi/2)$, which corresponds to choosing a generator $\alpha$ that lies within the first quadrant of the complex plane.\\
\\
A rational prime $p \in \Z$ \textit{splits} in $\Z[i]$ if and only if $p \equiv 1(4)$, in which case we factor $\langle p \rangle = \p_{1}\p_{2}$, where $N(\p_{1})=N(\p_{2}) = p$.  The primes $p \equiv 3(4)$ are \textit{inert}, i.e. $\langle p \rangle \subseteq \Z[i]$ remains prime in $\Z[i]$, and $N(\langle p \rangle )=p^{2}$.  The prime $2 = (1+i)^{2}$ is said to \textit{ramify} in $\Z[i]$, and $N(\langle 1+i \rangle)=2$.\\
\\
Consider the Hecke character
\begin{equation*}
\Xi_{k}(\A)= \left(\frac {\alpha}{ \overline{\alpha}}\right)^{2k} = e^{i 4k \theta_{\A}}, \hspace{5mm} k \in \Z,
\end{equation*}
and note that to each such character we may associate a degree 2 $L$-function
\begin{align*}
L_{k}(s) &= \ \sum_{\substack{\A \subseteq \Z[i]\\ \A \neq 0}}\frac{\Xi_{k}(\A)}{N(\A)^{s}}=\prod_{\p \textnormal{ prime}}\left(1-\frac{\Xi_{k}(\p)}{N(\p)^s}\right)^{-1}, \hspace{5mm} \textnormal{Re}(s)>1.
\end{align*}

Moreover, since 
\begin{equation}
\frac{a+bi}{|a+bi|}=\frac{|a-bi|}{a-bi},
\end{equation}
we have that
\begin{equation}
L_{k}(s)=\frac{1}{4}\sum_{\substack{{a,b \in \Z}\\ \{a,b\} \neq \{0,0\}}}\frac{\left(\frac{a+bi}{|a+bi|}\right)^{4k}}{(a^{2}+b^{2})^{s}}=\frac{1}{4}\sum_{\substack{{a,b \in \Z}\\ \{a,b\} \neq \{0,0\}}}\frac{\left(\frac{|a-bi|}{a-bi}\right)^{4k}}{(a^{2}+b^{2})^{s}}=L_{-k}(s)
\end{equation}
for any $k \in \Z$. Alternatively, we may choose to write $L_{k}(s)$ as a product over rational primes, namely
\begin{align*}
L_{k}(s) &=\prod_{p}L(s,k_{p}),
\end{align*}
where
\begin{align*}
L(s,k_{p})^{-1} &:= \left\{
\begin{array}{l l l}
\left(1-\Xi_{k}(\p_{1})p^{-s}\right)\left(1-\Xi_{k}(\p_{2})p^{-s}\right) & \text{ if } p=\p_{1}\p_{2}\equiv 1(4) \\
\big(1-p^{-s}\big)\big(1+p^{-s}\big) & \text{ if } p\equiv 3(4),\\
(1+(-1)^{k+1}2^{-s}) & \text{ if } p=2.
\end{array} \right.
\end{align*}

In terms of the \textit{local roots} $\alpha_{k}(p,j)$, this may be expressed as
\begin{align*}
L(s,k_{p}) = \prod_{j=1}^{2}\left(1-\frac{\alpha_{k}(p,j)}{p^{s}}\right)^{-1},
\end{align*}
where at the ramified prime $p=2$ we understand one of the two local roots to be zero. For Re$(s)>1$, we may then write

\begin{align}\label{Lk log derivative}
\frac{L_{k}'}{L_{k}}(s)&= -\sum_{n=1}^{\infty}\frac{\Lambda(n)a_{k}(n)}{n^{s}}= -\sum_{n=1}^{\infty}\frac{c_{k}(n)}{n^{s}},
\end{align}

where
\begin{equation*}
a_{k}(p^{l}) := \sum_{j=1}^{2}\alpha_{k}(p,j)^{l},
\end{equation*}
and
\begin{equation*}
c_{k}(n) := \Lambda(n)a_{k}(n).
\end{equation*}
Explicitly,
\begin{align*}
a_{k}(p^{l}) &= \left\{
\begin{array}{l l l l}
(-1)^{lk} & \text{ if } p =2\\
0 & \text{ if } p \equiv 3 (4), l \text{ is odd} \\
2 & \text{ if } p \equiv 3(4), l \text{ is even}\\
\Xi^{l}_{k}(\p_{1})+\Xi^{l}_{k}(\p_{2}) & \text{ if }p=\p_{1}\p_{2} \equiv 1(4),
\end{array} \right.
\end{align*}

while

\begin{align}\label{local von mangoldt junk}
c_{k}(p^{l}) &= \left\{
\begin{array}{l l}
(-1)^{lk}\log 2 & \text{ if } p=2 \\
0 & \text{ if } p \equiv 3 (4), l \text{ is odd} \\
2\cdot \log p & \text{ if } p \equiv 3(4), l \text{ is even}\\
(\Xi^{l}_{k}(\p_{1})+\Xi^{l}_{k}(\p_{2}))\log p & \text{ if }p=\p_{1}\p_{2} \equiv 1(4).
\end{array} \right.
\end{align}
Next, we note that $L_{k}(s)$ has conductor $Q_{k}=4$ and gamma factor

\begin{equation*}
L(s,k_{\infty}) = \prod_{j=1}^{2}\Gamma_{\R}\left(s+\kappa_{j}\right),
\end{equation*}
where $\kappa_{1} := |2k|$ and $\kappa_{2}:=|2k|+1$ are the \textit{local parameters} of $L_{k}(s)$ at infinity, and $
\Gamma_{\R}(s):= \pi^{-s/2}\Gamma\left(s/2\right).$  The \textit{completed $L$-function} is then given by

\begin{align*}
\Lambda_{k}(s) &:= Q_{k}^{\frac s 2}\cdot \prod_{j=1}^{2}\Gamma_{\R}\left(s+\kappa_{j}\right)\cdot L_{k}(s)\\
&\phantom{:}=\left(\frac{2}{\pi}\right)^{s} \pi^{-\left(\frac 1 2+|2k|\right)}\Gamma\left(\frac{s+|2k|}{2}\right)\Gamma\left(\frac{s+|2k|+1}{2}\right)L_{k}(s).
\end{align*}

Hecke showed that if $k\neq 0$, these functions have an analytic continuation to the entire complex plane.  The \textit{root number} of $L_{k}(s)$ is $\epsilon(k)=1$, which means that it satisfies the \textit{functional equation} $\Lambda_{k}(s) = \Lambda_{k}(1-s)$.  Moreover, by the Legendre Duplication Formula,
\begin{equation*}
\Gamma(s+|2k|) = \Gamma\left(\frac{s+|2k|}{2}\right)\Gamma\left(\frac{s+|2k|+1}{2}\right)\pi^{-\frac 1 2}2^{s+|2k|-1},
\end{equation*}
which allows us to represent $\Lambda_{k}(s)$ as
\begin{equation*}
\Lambda_{k}(s) = \Gamma(s+|2k|)L_{k}(s)2^{1-|2k|}\pi^{-s-|2k|},
\end{equation*}
and yields a functional equation of the form
\begin{equation}\label{functional equation}
L_{k}(s) = X_{k}(s)L_{k}(1-s),
\end{equation}
where
\begin{equation*}
X_{k}(s) := \pi^{2s-1}\frac{\Gamma(1-s+|2k|)}{\Gamma(s+|2k|)}.
\end{equation*}
$\Lambda_{k}(s)$ has all its zeros in the critical strip $0<\textnormal{Re}(s)<1$ (at the non-trivial zeros of $L_{k}(s)$), and the Generalized Riemann Hypothesis asserts that in fact they all lie on the critical line $\textnormal{Re}(s)=1/2$.
\section{Ratios Conjecture}
\subsection{The Recipe}
The \textit{Ratios Conjecture} described in \cite{Conrey} provides a procedure for calculating the ratio of a product of shifted $L$-functions averaged over a family.  Let $\mathcal{L}$ be an $L$-function, and $\mathcal{F} =\{f\}$ a family of characters with log conductor $c(f)$, as defined in section 3 of \cite{ConreyII}.  $\mathcal{L}(s,f)$ has an approximate functional equation given by
\begin{equation}\label{approx equation}
\mathcal{L}(s,f) = \sum_{n<x}\frac{A_{n}(f)}{n^s}+X(f,s)\sum_{m <y} \frac{\overline{A_{m}(f)}}{m^{1-s}}+remainder.
\end{equation}
Moreover, one may write
\begin{equation} \label{denominator}
\frac{1}{\mathcal{L}(s,f)} = \sum_{n=1}^{\infty}\frac{\mu_{f}(n)}{n^s},
\end{equation}
where the series converges absolutely for Re$(s)>1$.
To conjecture an asymptotic formula for the average
\begin{equation*}
\sum_{f \in \mathcal{F}}\frac{\mathcal{L}(\frac{1}{2} +\alpha,f)}{\mathcal{L}(\frac{1}{2}+ \gamma,f)}
\end{equation*}
the \textit{Ratios Conjecture} suggests the following recipe:
\begin{itemize}
\item Replace the $L$-function in the numerator with the two sums in its approximate functional equation, ignore the remainder term, and allow both sums to extend to infinity.  Replace the $L$-function in the denominator by its series (\ref{denominator}).  Multiply out the resulting expression.
\item Replace the $X(f,s)$ factor by its expected value when averaged over the family.
\item Replace each summand by its expected value when averaged over the family.
\item Call the total $R_{\mathcal{F}}(\alpha,\gamma)$, and let $F = |\mathcal{F}|$.  Then for 
\begin{equation}\label{ratios domain}
-\frac{1}{4}<\textnormal{Re}(\alpha)< \frac{1}{4}, \hspace{5mm}
\frac{1}{\log F}\ll \textnormal{Re}(\gamma),  \hspace{5mm}
\textnormal{Im}(\alpha),\textnormal{Im}(\gamma) \ll_{\epsilon}F^{1-\epsilon},
\end{equation}
the conjecture is that
\begin{equation*}
\sum_{f \in \mathcal{F}}\frac{\mathcal{L}(\frac{1}{2} + \alpha,f)}{\mathcal{L}(\frac{1}{2} + \gamma,f)}=\sum_{f \in \mathcal{F}}R_{\mathcal{F}}(\alpha,\gamma)\left(1+O\left(e^{(-\frac{1}{2}+\epsilon)c(f)}\right)\right),
\end{equation*}
for all $\epsilon > 0$.

\end{itemize}
\subsection{Approximate Functional Equation}
We seek to apply the above procedure to compute the average

\begin{equation}\label{R factor}
R_{K}(\alpha,\gamma)\ = \ \frac{1}{K}\sum_{k=1}^{K} \frac{L_{k}(\frac 1 2+\alpha)}{L_{k}(\frac 1 2+\gamma)},
\end{equation}
for specified values of $\alpha$ and $\gamma$.  As in \cite{SMALL paper}, the approximate functional equation is given by

\begin{equation}\label{approximate func}
L_{k}(s) = \sum_{n<x}\frac{A_{k}(n)}{n^s}+\pi^{2s-1}\cdot \frac{\Gamma(1-s+|2k|)}{\Gamma(s+|2k|)}\sum_{m<y}\frac{\overline{A_{k}}(m)}{m^{1-s}}+remainder,
\end{equation}
where 
\begin{align*}
A_{k}(n) &= \sum_{\substack{N(\A) = n}}\Xi_{k}(\A)
\end{align*}
is a multiplicative function defined explicitly on prime powers as
\begin{equation}\label{coefficients}
A_{k}(p^{l}) = \left\{
\begin{array}{l l}
\sum_{j=-l/2}^{l/2}\Xi^{2j}_{k}(\p) & \text{ if } p \equiv 1(4), l \text{ even}\\
\sum_{j=-(l+1)/2}^{(l-1)/2}\Xi^{2j+1}_{k}(\p) & \text{ if } p \equiv 1(4), l \text{ odd}\\
0 & \text{ if } p \equiv 3(4), l \textnormal{ odd }\\
1 & \text{ if } p \equiv 3(4), l \textnormal{ even } \\
(-1)^{lk} & \text{ if } p = 2.
\end{array} \right.
\end{equation}
Here $\p|p$, and we note that the above formula is independent of our choice of such $\p$.  Ignoring the remainder term, and allowing both terms in (\ref{approximate func}) to sum towards infinity, we write
\begin{equation*}
L_{k}(s) \approx \sum_{n=0}^{\infty}\frac{A_{k}(n)}{n^s}+\pi^{2s-1}\cdot \frac{\Gamma(1-s+|2k|)}{\Gamma(s+|2k|)}\sum_{m=0}^{\infty} \frac{A_{k}(m)}{m^{1-s}},
\end{equation*}
upon noting that $\overline{A_{k}(n)}=A_{k}(n)$ for all $A_{k}(n)$.  Again as in \cite{SMALL paper}, we note that

\begin{equation}\label{inverse series}
\frac{1}{L_k(s)}=\sum\limits_{h} \frac{\mu_{k}(h)}{h^s},
\end{equation}
where

\begin{align}\label{inverse coefficients}
\mu_k(p^h):=
\begin{cases}
1 & h=0\\
-A_{k}(p) & h=1\\
-1 & h=2,p\equiv 3(4)
\\
1 & h=2, p\equiv 1(4)
\\
0 & \textnormal{ otherwise}.
\end{cases}
\end{align}

Multiplying out the resulting expression yields
\begin{align*}
\frac{L_k(\frac{1}{2}+\alpha)}{L_k(\frac{1}{2}+\gamma)} \approx \begin{split}
\prod_{\substack{ p  }}\left(\sum_{n,h}\frac{\mu_{k}(p^h)A_{k}(p^{n})}{p^{h(\frac 1 2+\gamma)+n(\frac 1 2+\alpha)}}\right)+\pi^{2\alpha}\cdot \frac{\Gamma(\frac 1 2-\alpha+|2k|)}{\Gamma(\frac 1 2+\alpha+|2k|)}\prod_{\substack{ p  }}\left(\sum_{m,h}\frac{\mu_{k}(p^h)A_{k}(p^{m})}{p^{h(\frac 1 2+\gamma)+m(\frac 1 2-\alpha)}}\right).
\end{split}
\end{align*}

Next, set
\begin{equation*}
\mathbf{k}:= \frac{1}{2}+|2k|.
\end{equation*}
The algorithm now dictates that we compute an asymptotic for

\begin{equation*}
\bigg\langle \frac{\Gamma(\mathbf{k}-\alpha)}{\Gamma(\mathbf{k}+\alpha)}\bigg\rangle_{K} \hspace{5mm} \textnormal{and} \hspace{5mm}
\bigg \langle \prod_{\substack{ p  }}\left(\sum_{n,h}\frac{\mu_{k}(p^h)A_{k}(p^{n})}{p^{h(\frac 1 2+\gamma)+n(\frac 1 2+\alpha)}}\right)\bigg \rangle_{K},
\end{equation*}

where we denote $\langle \cdot \rangle_{K}$ to refer to the average over all $1 \leq k \leq K$.  
\subsection{Averaging the $\Gamma$-Factors}
\begin{lemma}\label{single gamma}
For $-1/4< \textnormal{Re}(\alpha)< 1/4,$ we have that
\begin{align*}
\bigg\langle \frac{\Gamma(\mathbf{k}-\alpha)}{\Gamma(\mathbf{k}+\alpha)}\bigg\rangle_{K} = \frac{\left(2K\right)^{-2\alpha}}{1-2\alpha} + O_{\alpha}\left(K^{-1}\right)+ O_{\alpha}\left(K^{-1-2\alpha}\right).
\end{align*}
\end{lemma}
\begin{proof}
By Stirling's approximation,

\begin{align*}
\begin{split}
\log \frac{\Gamma(\mathbf{k}-\alpha)}{\Gamma(\mathbf{k}+\alpha)}&= (\mathbf{k}-\alpha)\log (\mathbf{k}-\alpha)-(\mathbf{k}+\alpha)\log (\mathbf{k}+\alpha)+2\alpha\\
&\phantom{=}+\frac{1}{2}\left(\log (\mathbf{k}+\alpha)-\log (\mathbf{k}-\alpha)\right)+O\left(\frac{1}{\mathbf{k}}\right). \end{split}
\end{align*}
Taylor expanding the logarithm about $\mathbf{k}$ yields
\begin{equation*}
\log (\mathbf{k}-\alpha)=\log \mathbf{k}-\frac{\alpha}{\mathbf{k}}+O\left(\left(\frac{\alpha}{\mathbf{k}}\right)^{2}\right),
\end{equation*}
so that
\begin{align*}
\begin{split}
\log \frac{\Gamma(\mathbf{k}-\alpha)}{\Gamma(\mathbf{k}+\alpha)} &= (\mathbf{k}-\alpha)\left(\log \mathbf{k}-\frac{\alpha}{\mathbf{k}}\right)-(\mathbf{k}+\alpha)\left(\log \mathbf{k}+\frac{\alpha}{\mathbf{k}}\right)\\
&\phantom{=}+ 2\alpha+\frac{1}{2}\left(\log \mathbf{k}+\frac{\alpha}{\mathbf{k}}-\left(\log \mathbf{k}-\frac{\alpha}{\mathbf{k}}\right)\right)+O\left(\frac{1}{\mathbf{k}}\right)\\
&= -2\alpha\log \mathbf{k} + O\left(\frac{1}{\mathbf{k}}\right),
\end{split}
\end{align*}
and
\begin{align*}
\frac{\Gamma(\mathbf{k}-\alpha)}{\Gamma(\mathbf{k}+\alpha)}& = \mathbf{k}^{-2\alpha}\left(1+O\left(\frac{1}{\mathbf{k}}\right)\right).
\end{align*}

Suppose $f(x)$ is a continuous function for real numbers on the interval $[1,K]$.  Then by the Euler–Maclaurin Formula,
\begin{align*}
\sum_{k=1}^{K}f(k)-\int_{1}^{K}f(x)dx &= \frac{f(1)+f(K)}{2}+\int_{1}^{K}f'(x)\left(x-\lfloor x \rfloor - \frac{1}{2}\right)dx\\
&= O(1)+O(f(K))+O\left(\int_{1}^{K}|f'(x)|dx\right),
\end{align*}
in the limit as $K \rightarrow \infty$.  It follows that

\begin{align*}
\begin{split}
\frac {1}{K} &\sum_{k=1}^{K} \frac{\Gamma(\mathbf{k}-\alpha)}{\Gamma(\mathbf{k}+\alpha)} = \frac {1}{K} \int_1^K\left(\frac{1}{2}+2x\right)^{-2\alpha}dx+O_{\alpha}\left(\frac {1}{K} \int_1^K\left(\frac{1}{2}+2x\right)^{-2\cdot \textnormal{Re}(\alpha)-1}\right)dx\\
&\hspace{25mm} + O_{\alpha}\left(K^{-1}+K^{-2\alpha-1}\right)\\
& = \frac {1}{2K}\left[\frac{\left(\frac{1}{2}+2x\right)^{1-2\alpha}}{1-2\alpha}\right]_1^K +O_{\alpha}\left(\frac{\left(\frac 1 2+2K\right)^{-2 \alpha}}{K}\right) + O_{\alpha}\left(K^{-1}+K^{-2\alpha-1}\right)\\
& = \frac{(2K)^{-2\alpha}}{1-2\alpha}\left(\frac{\frac 1 2+2K}{2K}\right)^{1-2\alpha} + O_{\alpha}\left(K^{-1}+K^{-2\alpha-1}\right)\\
& = \frac{(2K)^{-2\alpha}}{1-2\alpha} + O_{\alpha}\left(K^{-1}+K^{-2\alpha-1}\right),
\end{split}
\end{align*}

as desired.
\end{proof}

\subsection{Averaging the Coefficients}
Define
\begin{equation*}
\delta_{p}(h,n) := \lim_{K \rightarrow \infty} \langle \mu_{k}(p^{h})A_{k}(p^{n}) \rangle_{K},
\end{equation*}
and write
\begin{align*}
\delta_{p}(h,n)\ =\ \begin{cases}
\delta_{3(4)}(h,n)  &\text{ when }p\equiv 3(4)\\
\delta_{1(4)}(h,n) &\text{ when }p\equiv 1(4)\\
\delta_{2}(h,n) &\text{ when }p=2.
\end{cases}
\end{align*}

\begin{lemma}
We have

\begin{align}\label{average 2}
\delta_{2}(h,n) =\left\{
\begin{array}{l l}
1 & \text{ if } h= 0, n \text{ is even}\\
-1 & \text{ if } h=1,  n \text{ is odd}\\
0 & \text{ otherwise},
\end{array} \right.
\end{align}

\begin{align}\label{average 1 mod 4}
\delta_{1(4)}(h,n) =\left\{
\begin{array}{l l}
1 & \text{ if } h= 0, n \text{ is even}\\
-2 & \text{ if } h=1,  n \text{ is odd}\\
1 & \text{ if } h= 2, n \text{ is even}\\
0 & \text{ otherwise},
\end{array} \right.
\end{align}

and

\begin{align}\label{average 3 mod 4}
\delta_{3(4)}(h,n) =\left\{
\begin{array}{l l}
1 & \text{ if } h= 0, n \text{ is even}\\
-1 & \text{ if } h= 2, n \text{ is even}\\
0 & \text{ otherwise}.
\end{array} \right.
\end{align}

\end{lemma}

\begin{proof}

The effect of $(\ref{inverse coefficients})$ is to limit the choices for $p^{h}$ to $h \in \{0, 1, 2\}$.  First, suppose $p=2$.  If $h=1$ and $n$ is even, then

\[\delta_{2}(h,n) = \lim_{K \rightarrow \infty} \frac{1}{K}\sum_{k=1}^{K} -A_{k}(p)= \lim_{K \rightarrow \infty} \frac{1}{K}\sum_{k=1}^{K} (-1)^{k+1} = 0,\]
while if $h=1$ and $n$ is odd, then
\[\delta_{2}(h,n) = \lim_{K \rightarrow \infty} \frac{1}{K}\sum_{k=1}^{K} -A_{k}(p)A_{k}(p^{n})= \lim_{K \rightarrow \infty} \frac{1}{K}\sum_{k=1}^{K} (-1)^{2k+1} = -1.\]

Proceeding similarly in the remaining cases, we obtain (\ref{average 2}).\\
\\
Next, suppose $p\equiv 1(4)$.  If $n$ is even and $h \in \{0,2\}$, then

\[\delta_{1(4)}(h,n) = \lim_{K \rightarrow \infty} \frac{1}{K}\sum_{k=1}^{K}\sum_{j=-n/2}^{n/2} \Xi_{k}^{2j}(\p) = \sum_{j=-n/2}^{n/2}\lim_{K \rightarrow \infty} \frac{1}{K}\sum_{k=1}^{K} \Xi_{2j}^{k}(\p) = 1,\]
upon noting that for fixed $\p|p$,
\begin{align}
\lim_{K \rightarrow \infty} \frac{1}{K}\sum_{k=1}^{K} \Xi_{2j}^{k}(\p) =\left\{
\begin{array}{l l}
1 & \text{ if } j= 0 \\
0 & \text{ otherwise}.
\end{array} \right.
\end{align}

Similarly, if $n$ is odd and $h \in \{0,2\}$, then

\[\delta_{1(4)}(h,n) = \sum_{j=-(n+1)/2}^{(n-1)/2}\lim_{K \rightarrow \infty} \frac{1}{K}\sum_{k=1}^{K} \Xi_{2j+1}^{k}(\p) = 0.\]
If $h=1$ and $n$ is even,
\begin{align*}
\delta_{1(4)}(h,n)&= \lim_{K \rightarrow \infty} \frac{1}{K}\sum_{k=1}^{K}\sum_{j=-n/2}^{n/2} -A_{k}(\p)\Xi_{k}^{2j}(\p) \\
&= -\lim_{K \rightarrow \infty} \sum_{j=-n/2}^{n/2}\frac{1}{K}\sum_{k=1}^{K} \bigg(\Xi_{-1}^{k}(\p)+\Xi_{1}^{k}(\p)\bigg)\Xi_{2j}^{k}(\p)\\
& = - \sum_{j=-n/2}^{n/2}\lim_{K \rightarrow \infty} \frac{1}{K}\sum_{k=1}^{K}\bigg( \Xi_{2j-1}^{k}(\p)+\Xi_{2j+1}^{k}(\p)\bigg) = 0,
\end{align*}

while if $h=1$ and $n$ is odd,
\begin{align*}
\delta_{1(4)}(h,n)&= -\lim_{K \rightarrow \infty} \sum_{j=-(n+1)/2}^{(n-1)/2}\frac{1}{K}\sum_{k=1}^{K} \bigg(\Xi_{-1}^{k}(\p)+\Xi_{1}^{k}(\p)\bigg)\Xi_{2j+1}^{k}(\p)\\
& = - \sum_{j=-(n+1)/2}^{(n-1)/2}\lim_{K \rightarrow \infty} \frac{1}{K}\sum_{k=1}^{K}\bigg( \Xi_{2j}^{k}(\p)+\Xi_{2j+2}^{k}(\p)\bigg) = -2.
\end{align*}

Calculating the remaining cases, we then obtain (\ref{average 1 mod 4}), and similarly for (\ref{average 3 mod 4}).
\end{proof}

Next, define
\begin{equation*}
G_{p}(\alpha,\gamma):=\sum_{h,n}\frac{\delta_{p}(h,n)}{p^{h(\frac 1 2+\gamma)+n(\frac 1 2+\alpha)}}.
\end{equation*}
If $p \equiv 3(4)$, we write

\begin{align}\label{prime 3(4)}
\begin{split}
G_{p}(\alpha,\gamma)&= \sum_{n \text{ even}}\frac{1}{p^{n(1/2+\alpha)}}-\sum_{n \text{ even}}\frac{1}{p^{1+2\gamma+n(1/2+\alpha)}}\\
&=\sum_{n=0}^{\infty}\left(\frac{1}{p^{n(1+2\alpha)}}-\frac{1}{p^{n(1+2\alpha)+1+2\gamma}}\right)\\
&= \left(1-\frac{1}{p^{1+2\gamma}}\right)\left(1-\frac{1}{p^{1+2\alpha}}\right)^{-1},
\end{split}
\end{align}

while if $p\equiv 1(4)$,

\begin{align}\label{prime 1(4)}
\begin{split}
G_{p}(\alpha,\gamma) &= \sum_{n}\left(\frac{\delta_{p}(0,n)}{p^{n(1/2+\alpha)}}+\frac{\delta_{p}(1,n)}{p^{1/2+\gamma+n(1/2+\alpha)}}+\frac{\delta_{p}(2,n)}{p^{1+2\gamma+n(1/2+\alpha)}}\right)\\
&= \sum_{n=0}^{\infty}\left(\frac{1}{p^{n(1+2\alpha)}}-\frac{1}{p^{n(1+2\alpha)}}\frac{2}{p^{1+\gamma+\alpha}}+\frac{1}{p^{n(1+2\alpha)}}\frac{1}{p^{1+2\gamma}}\right)\\
&= \left(1-\frac{2}{p^{1+\gamma+\alpha}}+\frac{1}{p^{1+2\gamma}}\right)\left(1-\frac{1}{p^{1+2\alpha}}\right)^{-1}.
\end{split}
\end{align}

Lastly, for $p=2$ we obtain

\begin{align}\label{prime 2}
\begin{split}
G_{2}(\alpha,\gamma)&=\sum_{n=0}^{\infty}\left(\frac{1}{2^{n(1+2\alpha)}}-\frac{1}{2^{1+\gamma+\alpha}}\frac{1}{2^{n(1+2\alpha)}}\right)\\
&=\left(1-\frac{1}{2^{1+\gamma+\alpha}}\right)\left(1-\frac{1}{2^{1+2\alpha}}\right)^{-1}.
\end{split}
\end{align}
Combining (\ref{prime 3(4)}), (\ref{prime 1(4)}), and (\ref{prime 2}), we obtain
\begin{align*}
G(\alpha,\gamma)&:=\prod_{p}G_{p}(\alpha,\gamma)\\
&=\left(1-\frac{1}{2^{1+\gamma+\alpha}}\right)\prod_{p\equiv 1(4)}\left(1-\frac{2}{p^{1+\gamma+\alpha}}+\frac{1}{p^{1+2\gamma}}\right)\\
&\phantom{=}\times \prod_{p\equiv 3(4)}\left(1-\frac{1}{p^{1+2\gamma}}\right)\prod_{p}\left(1-\frac{1}{p^{1+2\alpha}}\right)^{-1}.
\end{align*}
We conclude as follows:

\begin{conjecture}\label{ratios conjecture}
With constraints on $\alpha$ and $\gamma$ as described in \textnormal{(\ref{ratios domain})},
\begin{align*}
\begin{split}
R_{K}(\alpha,\gamma)&= G(\alpha,\gamma)+\frac{1}{1-2\alpha}\left(\frac{\pi}{2K}\right)^{2\alpha} G(-\alpha,\gamma)+O(K^{-\frac 1 2+\epsilon}),
\end{split}
\end{align*}
where $G(\alpha,\gamma)$ is as above.
\end{conjecture}
\subsection{Useful Properties of $G(\alpha,\gamma)$}
We next identify several properties of $G(\alpha,\gamma)$ relevant for our applications to the one-level density.  Let $L(s,\chi_{1})$ denote the Dirichlet $L$-function corresponding to the non-principal character $\chi_{1} \in (\Z/4\Z)^{\times}$. Assuming small fixed values of Re$(\alpha)$, Re$(\gamma)$, we factor out all terms which, for fixed $p$, decay substantially slower than $1/p^{2}$, and write
\begin{equation*}
G(\alpha,\gamma) = Y(\alpha,\gamma) \times A(\alpha,\gamma),
\end{equation*}
where
\begin{equation*}
Y(\alpha,\gamma):=\frac{L(1+2\gamma,\chi_{1})\zeta(1+2\alpha)}{ L(1+\alpha+\gamma,\chi_{1})\zeta(1+\alpha+\gamma)},
\end{equation*}
and
\begin{equation*}
A(\alpha,\gamma) := A_{3(4)}(\alpha,\gamma) \times A_{1(4)}(\alpha,\gamma),
\end{equation*}
where
\begin{equation*}
A_{3(4)}(\alpha,\gamma) := \prod_{p \equiv 3(4)}\frac{\left(1-\frac{1}{p^{1+2\gamma}}\right)\left(1+\frac{1}{p^{1+2\gamma}}\right)}{\left(1-\frac{1}{p^{1+\alpha+\gamma}}\right)\left(1+\frac{1}{p^{1+\alpha+\gamma}}\right)},
\end{equation*}
and
\begin{equation*}
A_{1(4)}(\alpha,\gamma) := \prod_{p \equiv 1(4)}\frac{\left(1-\frac{2}{p^{1+\alpha+\gamma}}+\frac{1}{p^{1+2\gamma}}\right)\left(1-\frac{1}{p^{1+2\gamma}}\right)}{\left(1-\frac{1}{p^{1+\alpha+\gamma}}\right)^{2}}.
\end{equation*}
In particular, note that $A(\alpha,\gamma)$ is differentiable at the point $\alpha = \gamma =0$.

\begin{lemma}\label{A derivative}
We have
\[\frac{\partial}{\partial \alpha} A(\alpha,\gamma)\bigg \vert_{\alpha = \gamma = r} = -2 \sum_{p \equiv 3(4)}\frac{\log p}{p^{2+4r}-1}\]
\end{lemma}
\begin{proof}
Let
\[B_{p}(\alpha,\gamma):= \frac{\left(1-\frac{1}{p^{1+2\gamma}}\right)\left(1+\frac{1}{p^{1+2\gamma}}\right)}{\left(1-\frac{1}{p^{1+\alpha+\gamma}}\right)\left(1+\frac{1}{p^{1+\alpha+\gamma}}\right)},\]

and note that 
\[\frac{\partial }{\partial \alpha}B_{p}(\alpha,\gamma) = \frac{\log p}{p^{1+\alpha+\gamma}}\left(\frac{\left(1-\frac{1}{p^{1+2\gamma}}\right)\left(1+\frac{1}{p^{1+2\gamma}}\right)}{\left(1-\frac{1}{p^{1+\alpha+\gamma}}\right)\left(1+\frac{1}{p^{1+\alpha+\gamma}}\right)^{2}}-\frac{\left(1-\frac{1}{p^{1+2\gamma}}\right)\left(1+\frac{1}{p^{1+2\gamma}}\right)}{\left(1-\frac{1}{p^{1+\alpha+\gamma}}\right)^{2}\left(1+\frac{1}{p^{1+\alpha+\gamma}}\right)}\right).\]
It follows that
\begin{align*}
\frac{\partial }{\partial \alpha}A_{3(4)}(\alpha,\gamma)\bigg \vert_{\alpha = \gamma = r} &= \sum_{p \equiv 3(4)}\prod_{\substack{q \neq p \\ q \equiv 3(4)}}B_{q}(\alpha,\gamma)\frac{\partial }{\partial \alpha}B_{p}(\alpha,\gamma)\bigg \vert_{\alpha = \gamma = r}\\
&= \sum_{p \equiv 3(4)}\frac{\partial }{\partial \alpha}B_{p}(\alpha,\gamma)\bigg \vert_{\alpha = \gamma = r}\\
&=  \sum_{p \equiv 3(4)}\frac{\log p}{p^{1+2r}}\left(\left(1+\frac{1}{p^{1+2r}}\right)^{-1}-\left(1-\frac{1}{p^{1+2r}}\right)^{-1}\right)\\
&=  -2\sum_{p \equiv 3(4)}\frac{\log p}{p^{2+4r}-1}.
\end{align*}
Similarly, we find that
\begin{align*}
\frac{\partial }{\partial \alpha}A_{1(4)}(\alpha,\gamma)\bigg \vert_{\alpha = \gamma = r} &= 2\sum_{p \equiv 1(4)}\frac{\log p}{p^{1+2r}}\Bigg(\left(1 - \frac{1}{p^{1+2r}}\right)^{-1}-\left(1 - \frac{1}{p^{1+2r}}\right)^{-1}\Bigg)=0.
\end{align*}

Since $A_{3(4)}(r,r) = A_{1(4)}(r,r) = 1$, it follows that

\begin{align*}
\frac{\partial }{\partial \alpha}A(\alpha,\gamma)\bigg \vert_{\alpha=\gamma=r} &= \frac{\partial }{\partial \alpha}A_{3(4)}(\alpha,\gamma)\bigg \vert_{\alpha=\gamma=r}+\frac{\partial }{\partial \alpha}A_{1(4)}(\alpha,\gamma)\bigg \vert_{\alpha=\gamma=r}\\
&= -2\sum_{p \equiv 3(4)}\frac{ \log p}{p^{2+4r}-1}.
\end{align*}
\end{proof}
\begin{lemma}\label{A derivative at zero}
We have 
\[\frac{d}{d \alpha}A(-\alpha,\alpha)\bigg\vert_{\alpha=0}=4\sum_{p \equiv 3(4)}\frac{\log p}{p^{2}-1}\]
\end{lemma}
\begin{proof}
Note that
\[\frac{d }{d r }A_{3(4)}(-r,r) = \frac{d }{d r }\prod_{p \equiv 3(4)}\frac{\left(1-\frac{1}{p^{1+2r}}\right)\left(1+\frac{1}{p^{1+2r}}\right)}{\left(1-\frac{1}{p}\right)\left(1+\frac{1}{p}\right)} = \sum_{p \equiv 3(4)}\frac{4\log p}{p^{4r}(p^{2}-1)},\]
while
\begin{align*}
\frac{d }{d r} A_{1(4)}(-r,r)&=\frac{d }{d r}\prod_{p \equiv 1(4)}\frac{\left(1-\frac{2}{p}+\frac{1}{p^{1+2r}}\right)\left(1-\frac{1}{p^{1+2r}}\right)}{\left(1-\frac{1}{p}\right)^{2}}\\
& = \sum_{p \equiv 1(4)}-\frac{4(p^{2r}-1)\log p}{p^{4r}(p-1)^{2}}.
\end{align*}

It follows that
\begin{align*}
\frac{d}{d r}A(-r,r)\bigg \vert_{r=0}&= \frac{d}{d r}A_{3(4)}(-r,r)\bigg \vert_{r=0}+\frac{d}{d r}A_{1(4)}(-r,r)\bigg \vert_{r=0}\\
&=4\sum_{p \equiv 3(4)}\frac{\log p}{p^{2}-1}.
\end{align*}
\end{proof}
\subsection{Taking the Derivative}
Let $R_{K}(\alpha,\gamma)$ be as in (\ref{R factor}). Since differentiation operates linearly on finite sums,
upon setting $\gamma = \alpha=r$ we obtain
\begin{align}\label{ratiostoder}
\begin{split}
\frac{1}{K}\sum_{k=1}^{K}\frac{L_{k}'}{L_{k}}\left(\frac{1}{2}+r\right)=\frac{\partial}{\partial \alpha}R_{K}(\alpha,\gamma)\bigg \vert_{\gamma = \alpha=r} .
\end{split}
\end{align}

By Conjecture \ref{ratios conjecture},
\begin{align*}
\begin{split}
\frac{\partial}{\partial \alpha}R_{K}(\alpha;\gamma)\bigg|_{\alpha=\gamma=r}&= \frac{1}{K}\sum_{k=1}^{K}\Bigg(\frac{\partial}{\partial \alpha}G(\alpha,\gamma)+\frac{\partial }{\partial \alpha}\frac{1}{1-2\alpha}\left(\frac{\pi}{2K}\right)^{2\alpha} G(-\alpha,\gamma)\Bigg)\bigg|_{\alpha=\gamma=r}
\end{split}
\end{align*}
up to an error term, and moreover, upon noting that $A(r,r)=1$,
\begin{equation*}
\frac{\partial }{\partial \alpha}G(\alpha,\gamma)\bigg \vert_{\alpha= \gamma = r}=\frac{\zeta'}{\zeta}(1+2r)-\frac{L'}{L}(1+2r,\chi_{1})+A'(r,r)
\end{equation*}
where $A'(r,r):= \frac{\partial }{\partial \alpha}A(\alpha,\gamma)\bigg \vert_{\alpha= \gamma = r}$. Similarly,
\begin{align*}
\begin{split}
\frac{\partial }{\partial \alpha}\frac{1}{1-2\alpha}\left(\frac{\pi}{2K}\right)^{2\alpha} G(-\alpha,\gamma)\bigg \vert_{\alpha= \gamma = r}&= \frac{-1}{1-2r}\left(\frac{\pi}{2K}\right)^{2r}\frac{L(1+2r,\chi_{1})}{L(1,\chi_{1})} \zeta(1-2r)A(-r,r).
\end{split}
\end{align*}

The error terms in Conjecture \ref{ratios conjecture} is believed to be uniform in $\alpha$ and $\gamma$ (across the ranges permitted by (\ref{ratios domain})), and we expect that the results remain valid with the same range and error terms upon differentiating with respect to $\alpha$ \cite[Remark 2.4]{ConreySnaith}.  We therefore conclude as follows.
\begin{theorem}\label{log der ratios}
Let $1/\log K < \textnormal{Re}(r) < 1/4$ and $\textnormal{Im}(r)\ll_{\epsilon}K^{1-\epsilon}$.  Assume Conjecture \ref{ratios conjecture}, and that the error term remains invariant under differentiation with respect to $\alpha$.  We then have that
\begin{align*}
&\frac{1}{K}\sum_{k=1}^{K}\frac{L'_{k}}{L_{k}}\left(\frac 1 2+r\right)= \frac{\zeta'}{\zeta}(1+2r)+A'(r,r)-\frac{L'}{L}(1+2r,\chi_{1})\\
&-\frac{1}{1-2r}\left(\frac{\pi}{2K}\right)^{2r}\frac{L(1+2r,\chi_{1})}{L(1,\chi_{1})} \zeta(1-2r)A(-r,r)+O(K^{-\frac{1}{2}+\epsilon}).
\end{align*}
\end{theorem}

We conclude this section with the following technical lemma:

\begin{lemma}\label{fourierdecay}
Let $\widehat{f}$ be a smooth compactly supported function.  Then $f(\sigma+it)$ decays rapidly on horizontal strips, i.e. for fixed $t \in \R$,
\begin{equation}\label{rapiddecay}
f(\sigma+it) \ll_{f,n,t} \text{min}(1,|\sigma|^{-n})
\end{equation}
for any given $n > 0$.

\end{lemma}
\begin{proof}
Set $s:= \sigma+it$.  Upon applying integration by parts, we find that
\begin{align*}
f(s) = \int_{\R}\widehat{f}(r)e^{2\pi i r s} dr &= \widehat{f}(r)\frac{e^{2 \pi i r s}}{2\pi i s}\bigg|_{- \infty}^{\infty} - \int_{\R}\widehat{f}^{(1)}(r)\frac{e^{2 \pi i r s}}{2\pi i s}dr\\
&= - \frac{1}{2\pi i s}\int_{\R}\widehat{f}^{(1)}(r)e^{2 \pi i r s}dr.
\end{align*}
Similarly,
\begin{align*}
\int_{\R}\widehat{f}^{(1)}(r)e^{2 \pi i r s}dr &= \widehat{f}^{(1)}(r)\frac{e^{2 \pi i r s}}{2\pi i s}\bigg|_{- \infty}^{\infty} - \int_{\R}\widehat{f}^{(2)}(r)\frac{e^{2 \pi i r s}}{2\pi i s}dr\\
&= - \frac{1}{2\pi i s}\int_{\R}\widehat{f}^{(2)}(r)e^{2 \pi i r s}dr.
\end{align*}
Proceeding in this manner, we find that
\begin{align*}
f(s) &= \left(\frac{-1}{2\pi i s}\right)^{n}\int_{\R}\widehat{f}^{(n)}(r)e^{2 \pi i r s}dr
\end{align*}
for any $n \in \N$.  Since $\widehat{f}^{(n)}(r)$ is compactly supported, we bound
\begin{align*}
\int_{\R}\widehat{f}^{(n)}(r)e^{2 \pi i r (\sigma+it)}dr & = O(1),\\
\end{align*}
in the limit as $\sigma \rightarrow \infty$.  It follows that
\begin{equation}
f(\sigma+it) \ll \text{min}(1,|\sigma|^{-n}),
\end{equation}

where the implicit constant now depends on $f,n,$ and $t$.  The lemma then follows upon noting that this holds for any $n \in \N$.
\end{proof}
\section{One-Level Density Using the Ratios Conjecture}
In this section, we use the Ratios Conjecture to compute $D_{1}(\mathcal{F}(K);f)$.  Let $f$ be an even Schwarz function that is extended holomorphically to the strip $|$Im$(z)|<2$, and recall that $M= \log K$.  By an application of Cauchy's residue theorem (and the assumption of GRH), we find that
\begin{equation}\label{cauchy one level}
D_{1}(\mathcal{F}(K);f) = \frac{1}{2\pi i}\frac{1}{K}\sum_{k=1}^{K}\left(\int_{(c)}-\int_{(1-c)}\right)\frac{L'_{k}}{L_{k}}(s)\cdot f\left(\frac{-iM}{\pi}\left(s-\frac 1 2\right)\right)ds,
\end{equation}
where $\textnormal{Re}(s)=c$, and we chose $1/2+M^{-1} < c < 3/4$.  Here $(c)$ denotes the vertical line from $c-i\infty$ to $c+i \infty$.\\
\\
For the integral along the line $(1-c)$ we apply the change of variables $s \mapsto 1-s$, and find that this integral equals 
\begin{equation*}
\frac{1}{2\pi i}\frac{1}{K}\sum_{k=1}^{K}\int_{(c)}f\left(\frac{i M}{\pi}\left(s-\frac{1}{2}\right)\right)\frac{L'_{k}}{L_{k}}\left(1-s\right)ds.
\end{equation*}

By (\ref{functional equation}), we may moreover write

\begin{equation}\label{functional X}
\frac{L_{k}'}{L_{k}}(1-s) = \frac{X_{k}'}{X_{k}}(s)-\frac{L_{k}'}{L_{k}}(s),
\end{equation}
where
\begin{equation}\label{logarithmic derivative func}
\frac{X'_{k}}{X_{k}}(s) = 2 \textnormal{ log } \pi-\frac{\Gamma'}{\Gamma}\left(1-s+|2k|\right)-\frac{\Gamma'}{\Gamma}\left(s+|2k|\right).
\end{equation}
Changing the order of summation and integration and applying the change of variables $s = 1/2+r$, we then obtain
\begin{equation}\label{one level}
D_{1}(\mathcal{F}(K);f) = \frac{1}{2\pi i}\int_{(c')}\frac{1}{K}\sum_{k=1}^{K}\Bigg(2\frac{L'_{k}}{L_{k}}\left(\frac 1 2+r\right)-\frac{X'_{k}}{X_{k}}\left(\frac 1 2+r\right)\Bigg)f\left(\frac{i M r}{\pi}\right)dr,
\end{equation}
where Re$(r) = c-1/2 = c'$, and where we recall that $f$ is even.  When Im$(r)> |K|^{1-\epsilon}$, then by Lemma \ref{fourierdecay}, Stirling's approximation, and upper bounds on the growth of $L'_{k}/L_{k}$ within the critical strip, we bound the tail of this integral by $O_{\epsilon}\left(K^{\epsilon-1}\right)$.  For Im$(r) < K^{1-\epsilon}$, we apply Theorem \ref{log der ratios}, and a similar argument as above to bound the tail, to then obtain

\begin{align*}
\begin{split}
D_{1}(\mathcal{F}(K);f) &=\frac{-1}{2\pi i}\int_{(c')}f\left(\frac{i M r}{\pi}\right)\Bigg(\bigg \langle\frac{X'_{k}}{X_{k}}\left(\frac 1 2 +r\right)\bigg \rangle_{K}-2\frac{\zeta'}{\zeta}(1+2r)\\
&\phantom{=}+2\frac{L'}{L}(1+2r,\chi_{1})-2A'(r,r)\\
&\phantom{=}+2\frac{A(-r,r)}{1-2r}\left(\frac{\pi}{2K}\right)^{2r}\frac{L(1+2r,\chi_{1})}{L(1,\chi_{1})} \zeta(1-2r)\Bigg)dr+O_{\epsilon}\left(K^{-\frac 1 2+\epsilon}\right).
\end{split}
\end{align*}
Upon applying the change of variables $\pi i \tau/M = r$, we arrive at the following conjecture.

\begin{conjecture}\label{1-c line}
\begin{equation*}
D_{1}(\mathcal{F}(K);f) = W_{f}+S_{\zeta}+S_{L}+S_{A'}+S_{\Gamma}+O\left(K^{-\frac 1 2+\epsilon}\right),
\end{equation*}
where

\begin{align}\label{def Wf}
W_{f} &:-\frac{1}{2M}\int_{(\mathcal{C}')}f\left(\tau \right)\bigg \langle\frac{X'_{k}}{X_{k}}\left(\frac 1 2 +\frac{\pi i \tau}{M}\right)\bigg \rangle_{K}d\tau,
\end{align}

\begin{align}
S_{\zeta}&:=\frac{1}{M}\int_{(\mathcal{C}')}f(\tau)\frac{\zeta'}{\zeta}\left(1+\frac{2\pi i \tau}{M}\right)d\tau,
\end{align}

\begin{align}\label{L function}
S_{L}&:=-\frac{1}{M}\int_{(\mathcal{C}')}f(\tau)\frac{L'}{L}\left(1+\frac{2\pi i \tau}{M},\chi_{1}\right)d\tau,
\end{align}

\begin{equation}
S_{A'}:=\frac{1}{M}\int_{(\mathcal{C}')}f(\tau)A'\left(\frac{\pi i \tau}{M},\frac{\pi i \tau}{M}\right)d\tau,
\end{equation}

\begin{equation}
S_{\Gamma}:=-\frac{1}{M}\int_{(\mathcal{C}')}f(\tau)\frac{A\left(-\frac{\pi i \tau}{M},\frac{\pi i \tau}{M}\right)}{1-\frac{2\pi i \tau}{M}}\left(\frac{\pi}{2K}\right)^{\frac{2\pi i \tau}{M}}\frac{L\left(1+\frac{2\pi i \tau}{M},\chi_{1}\right)}{L(1,\chi_{1})}\zeta\left(1-\frac{2\pi i \tau}{M}\right)d\tau,
\end{equation}

and $(\mathcal{C}')$ denotes the horizontal line Im$(\tau) = - Mc'/\pi$.
\end{conjecture}
Before directly computing $W_{f}$, $S_{\zeta}$, $S_{A'}$, $S_{L}$, and $S_{\Gamma}$, we note the following lemma, which will be relevant for our work in Section 6.

\begin{lemma}\label{zeta and L as sum}
We have
\begin{equation}\label{zeta as sum}
S_{\zeta} = - \frac{1}{M}\sum_{n=1}^{\infty}\frac{\Lambda(n)}{n}\widehat{f}\left(\frac{\log n}{M}\right),
\end{equation}
\begin{equation}\label{L as sum}
S_{L} = \frac{1}{M}\sum_{n=1}^{\infty}\frac{\Lambda(n)\chi_{1}(n)}{n}\widehat{f}\left(\frac{\log n}{M}\right),
\end{equation}
and
\begin{equation}\label{A as sum}
S_{A'} = -\frac{2}{M}\sum_{p\equiv 3(4)}\sum_{n=1}^{\infty}\frac{\log p}{p^{2n}}\widehat{f}\left(\frac{2n \textnormal{ log }p}{M}\right).
\end{equation}

\end{lemma}

\begin{proof}
Write
\begin{align*}
S_{\zeta}&= -\frac{1}{M}\sum_{n}\frac{\Lambda(n)}{n}\int_{(\mathcal{C'})}f(\tau)e^{-2\pi i \tau\left(\frac{\log n}{M}\right)}d\tau,
\end{align*}
where we apply Lemma \ref{fourierdecay}, and note that interchanging the order of integration and summation is justified since the summation inside the integral converges absolutely and uniformly on compact sets.  Since the integrand is an entire function, we may shift the contour from $\mathcal{C'}$ to the line Im$(\tau) = 0$, from which (\ref{zeta as sum}) then follows.\\
\\
Equation (\ref{L as sum}) follows from an identical argument, upon noting that for Re$(z)>1$, we may write
\begin{equation*}
\frac{L'}{L}(z,\chi_{1}) = -\sum_{n=1}^{\infty}\frac{\Lambda(n)\chi_{1}(n)}{n^z}.
\end{equation*}

Finally, to compute $S_{A'}$, we note that
\begin{align*}
S_{A'}&= -\frac{2}{M}\sum_{p \equiv 3(4)}\sum_{n=1}^{\infty}\frac{\log p}{p^{2n}}\int_{(\mathcal{C}')}e^{- (2\pi i \tau)(\frac{2 n \log p}{M})}f(\tau)d\tau
\end{align*}
by Lemma \ref{A derivative}.  (\ref{A as sum}) then follows upon shifting the contour to the line Im$(\tau) = 0$.
\end{proof}

\section{Explicit Computations}
In this section, we explicitly compute $W_{f}$, $S_{\zeta}$, $S_{A'}$, $S_{L}$, and $S_{\Gamma}$, resulting in Conjecture \ref{ratios prediction}.
\subsection{Computing $W_{f}$:} 
\begin{lemma}\label{computing Wf} We have
\begin{align*}
W_{f}=  \widehat{f}(0) +\frac{1}{M}(\log 2-1-\log \pi)\widehat{f}(0)+O\left(\frac{1}{K}\right).
\end{align*}
\end{lemma}
\begin{proof}

Since $X'_{k}/X_{k}$ is holomorphic within the horizontal strip $-Mc'/ \pi \leq \textnormal{Im}(\tau) \leq 0$,  we may shift the line of integration to Im$(\tau) = 0$, and write 
\begin{align*}
W_{f} &\phantom{:}=\frac{1}{2M}\int_{\R}f(\tau)\bigg(\bigg \langle\frac{\Gamma'}{\Gamma}\left(\mathbf{k}+\frac{\pi i \tau}{M}\right)+\frac{\Gamma'}{\Gamma}\left(\mathbf{k}-\frac{\pi i \tau}{M}\right)\bigg \rangle_{K}-2\textnormal{ log }\pi\bigg)d\tau.\nonumber 
\end{align*}
By Stirling's approximation, for any $\textbf{k} \in \R$ and $\alpha \in i \R$,
\begin{align*}
\frac{\Gamma'}{\Gamma}\bigg(\mathbf{k}+\alpha\bigg)+\frac{\Gamma'}{\Gamma}\bigg(\mathbf{k}-\alpha\bigg) &= \log |\mathbf{k}+\alpha|+\log |\mathbf{k}-\alpha|+O\left(\frac{1}{|\mathbf{k}|+|\alpha|}\right)\\
&= 2 \cdot \log |\mathbf{k}+\alpha|+O\left(\frac{1}{|\mathbf{k}|+|\alpha|}\right).
\end{align*}

Thus

\begin{align*}
W_{f} &=\frac{1}{M}\int_{\R}f(\tau)\bigg( \bigg \langle \log \left|\mathbf{k}+\frac{\pi i \tau}{M}\right|\bigg \rangle_{K}-\textnormal{ log }\pi\bigg)d\tau+O\left(\frac{1}{K}\right),
\end{align*}

where the error term is calculated upon noting that
\[O\left(\frac{1}{K}\sum_{k=1}^{K}\frac{1}{|k|}\int_{\R}f(\tau)d\tau \right) = O\left(\frac{M}{K}\right).\]

Next, write
\begin{align*}
\int_{\R}f\left(\tau\right)\left(\textnormal{log }\left|\mathbf{k}+\frac{\pi i \tau}{M}\right|\right)d\tau &= \int_{\R}f(\tau)\left(\textnormal{log }\left|\mathbf{k}\left(1+\frac{\pi i \tau}{\mathbf{k}\cdot M}\right)\right|\right)d\tau,\\
&= \widehat{f}(0)\cdot \log \mathbf{k}+\int_{\R}f(\tau)\left(\textnormal{log }\left|\left(1+\frac{\pi i \tau}{\mathbf{k}\cdot M}\right)\right|\right)d\tau,
\end{align*}
and split
\begin{align*}
\begin{split}
\int_{\R}f(\tau)\textnormal{log }\left|\left(1+\frac{\pi i \tau}{\mathbf{k}\cdot M}\right)\right|d \tau&= \bigg( \int_{|\tau| \leq \sqrt{M}}+\int_{|\tau| \geq \sqrt{M}}\bigg) f(\tau)\textnormal{log }\left|\left(1+\frac{\pi i \tau}{\mathbf{k}\cdot M}\right)\right|d\tau.
\end{split}
\end{align*}

Since for $|z|<1$ we have $\log |1+z| = |z| + O(z^{2})$, we bound
\begin{equation*}
\int_{|\tau| \leq \sqrt{M}}f(\tau)\textnormal{log }\left|\left(1+\frac{\pi i \tau}{\mathbf{k}\cdot M}\right)\right|d\tau = O\left(\frac{1}{\mathbf{k}\sqrt{M}}\right),
\end{equation*}
and by the rapid decay of $f$,
\begin{align*}
\int_{|\tau| \geq \sqrt{M}}f(\tau)\textnormal{log }\left|\left(1+\frac{\pi i \tau}{\mathbf{k}\cdot M}\right)\right|d\tau&= O\bigg (\int_{|\tau| \geq \sqrt{M}}f(\tau)\textnormal{log}\left|\tau\right|d\tau \bigg )= O\bigg (\frac{1}{M^{100}} \bigg ).
\end{align*}
Thus
\begin{equation*}
\int_{\R}f(\tau)\bigg(\log \bigg|\mathbf{k}+\frac{\pi i \tau}{M} \bigg| \bigg)d\tau = \widehat{f}(0)\cdot \log |\mathbf{k}|+O\left(\frac{1}{k \cdot \sqrt{M}}\right),
\end{equation*}
so that
\begin{align}\label{Wf approx}
W_{f}&=\frac 1 M \cdot \widehat{f}(0) \bigg(\frac{1}{K}\sum_{k=1}^{K}\log |\mathbf{k}|-\log \pi\bigg)+O\left(\frac {1}{K}\right).
\end{align}

Finally, since
\begin{equation*}
\sum_{k=1}^{K}\log |2k|\leq \sum_{k=1}^{K}\log |\mathbf{k}| \leq\sum_{k=2}^{K+1}\log |2k|, 
\end{equation*}
we have
\begin{align}\label{stirling}
\begin{split}
\sum_{k=1}^{K}\log |\mathbf{k}| &= \sum_{k=1}^{K}\left(\log k+\log 2\right)+O(\log K) \\
&= \log K!+(\log 2)K+O(M)\\
&= K\cdot M+(\log 2-1)K+O(M).
\end{split}
\end{align}
Lemma \ref{computing Wf} now follows from (\ref{Wf approx}) and (\ref{stirling}). 
\end{proof}
\subsection{Computing $S_{\zeta}$:}
\begin{lemma}\label{Szeta}
Fix $J \in \N$.  We have
\begin{equation}
S_{\zeta} = - \frac{f(0)}{2}-\sum_{j=1}^{J}\frac{c_{j}\widehat{f}^{(j-1)}(0)}{M^{j}}+O_{J}\left(M^{-(J+1)}\right),
\end{equation}
where 
\begin{equation}\label{c1}
c_{1}:=\int_{1}^{\infty}\frac{\psi(t)-t}{t^{2}}dt+1,
\end{equation}
and for $j \geq 2$,
\begin{equation}
c_{j}:=\frac{1}{(j-2)!}\int_{1}^{\infty}(\log t)^{j-2}\left(\frac{\log t}{j-1}-1\right)\frac{\psi(t)-t}{t^{2}}dt,
\end{equation}
where $\psi(t):= \sum_{n\leq t}\Lambda(n)$ is the second Chebyshev function.
\end{lemma}
\begin{proof}
The methods employed in this proof are essentially equivalent to those found in Section 4 of \cite{DFiorPark}.  Note that
\begin{align*}
&\frac{1}{M}\sum_{n=1}^{\infty}\frac{\Lambda(n)}{n}\widehat{f}\left(\frac{\log n}{M}\right) = \frac{1}{M}\int_{1}^{\infty}\frac{1}{t}\widehat{f}\left(\frac{\log t}{M}\right)d \psi(t)\\
&= \frac{1}{M}\int_{1}^{\infty}\frac{1}{t}\widehat{f}\left(\frac{\log t}{M}\right)d \psi(t)-\frac{1}{M}\int_{1}^{\infty}\frac{1}{t}\widehat{f}\left(\frac{\log t}{M}\right)dt+\frac{1}{M}\int_{1}^{\infty}\frac{1}{t}\widehat{f}\left(\frac{\log t}{M}\right)dt.
\end{align*}
Setting $u:=\frac{\log t}{M}$, we have that 
\begin{align*}
\frac{1}{M}\int_{1}^{\infty}\frac{1}{t}\widehat{f}\left(\frac{\log t}{M}\right)dt = \int_{0}^{\infty}\widehat{f}\left(u\right)du = \frac{f(0)}{2},
\end{align*}
Integrating by parts, we moreover find that
\begin{align*}
\int_{1}^{\infty}\frac{1}{t}\widehat{f}\left(\frac{\log t}{M}\right)d \psi(t) &= \frac{\psi(t)}{t}\widehat{f}\left(\frac{\log t}{M}\right)\bigg|_{1}^{\infty} - \int_{1}^{\infty}\bigg[\frac{1}{t}\widehat{f}\left(\frac{\log t}{M}\right) \bigg]' \psi(t)dt\\
&= - \int_{1}^{\infty}\bigg[\frac{1}{t}\widehat{f}\left(\frac{\log t}{M}\right) \bigg]' \psi(t)dt,
\end{align*}
and similarly that
\begin{align*}
\int_{1}^{\infty}\frac{1}{t}\widehat{f}\left(\frac{\log t}{M}\right)dt &= \widehat{f}\left(\frac{\log t}{M}\right)\bigg|_{1}^{\infty} - \int_{1}^{\infty}\bigg[\frac{1}{t}\widehat{f}\left(\frac{\log t}{M}\right) \bigg]' t dt\\
&= -\widehat{f}(0)- \int_{1}^{\infty}\bigg[\frac{1}{t}\widehat{f}\left(\frac{\log t}{M}\right) \bigg]' t dt.
\end{align*}
It follows that
\begin{align*}
S_{\zeta} &= -\frac{f(0)}{2}-\frac{\widehat{f}(0)}{M} -\frac{1}{M} \int_{1}^{\infty}\left[\frac{1}{M}\widehat{f}'\left(\frac{\log t}{M}\right)-\widehat{f}\left(\frac{\log t}{M}\right) \right] \frac{\psi(t)-t}{t^{2}}dt.
\end{align*}

By the prime number theorem
\[\psi(t)-t \ll t \cdot \textnormal{exp}(-2c\sqrt{\log t}),\]

so that for any $0 < \xi < 1$,
\[\int_{K^{\xi/2}}^{\infty}\left[\frac{1}{M}\widehat{f}'\left(\frac{\log t}{M}\right)-\widehat{f}\left(\frac{\log t}{M}\right) \right] \frac{\psi(t)-t}{t^{2}}dt \ll \textnormal{exp}(-c \sqrt{\xi M})\]
upon recalling that $M:= \log K$.  Upon Taylor expanding about $\widehat{f}(0)$ we then find that

\begin{align*}
&-\frac{1}{M} \int_{1}^{K^{\frac{\xi}{2}}}\left[\frac{1}{M}\widehat{f}'\left(\frac{\log t}{M}\right)-\widehat{f}\left(\frac{\log t}{M}\right) \right] \frac{\psi(t)-t}{t^{2}}dt\\
&=\sum_{j=0}^{J}\frac{1}{j!}\left(\widehat{f}^{(j)}(0)-\frac{\widehat{f}^{(j+1)}(0)}{M}\right)\int_{1}^{K^{\frac{\xi}{2}}}\frac{(\log t)^{j}}{M^{j+1}} \frac{\psi(t)-t}{t^{2}}dt\\
&\phantom{=}+O_{J}\left(\int_{1}^{K^{\frac{\xi}{2}}}\frac{(\log t)^{J+1}}{M^{J+2}}\frac{\psi(t)-t}{t^{2}}dt\right)\\
&=\sum_{j=1}^{J+1}\frac{c_{j}\widehat{f}^{j-1}(0)}{M^{j}}+O_{J}\left(M^{-(J+2)}+e^{-c\sqrt{\xi M}}\right).
\end{align*}
The result then follows upon choosing $\xi = M^{-1+\delta}$ for some $\delta >0$.
\end{proof}

\subsection{Computing $S_{L}$ and $S_{A'}$:}
To compute $S_{L}$ we require the following lemma:

\begin{lemma}\label{special value} We have
\begin{align*}
\frac{L'}{L}(1,\chi_{1}) &= \bigg(\gamma_{0}-2\log 2-4\log |\eta(i)|\bigg),
\end{align*}
where 
\begin{equation}\label{chowla}
\eta(i)=e^{-\frac{\pi}{12}}\prod_{n=1}^\infty \bigg(1-e^{- 2\pi n}\bigg) = \frac{\Gamma\left(\frac 1 4\right)}{2\pi^{3/4}}.
\end{equation}
\end{lemma}

\begin{proof}
Upon inserting $\tau = i$ into the Kronecker First Limit Formula \cite{Lang}, we obtain
\begin{equation*}
\frac{1}{4}\sum_{(m,n)\ne (0,0)}\frac{1}{(m^2+n^2)^{s}} = \frac{\pi/4}{s-1}+\frac{\pi}{2} \bigg(\gamma_{0}-\log 2-2\log |\eta(i)|\bigg)+O(s-1),
\end{equation*}
where
\begin{equation*}
\eta(\tau) = e^{i\frac{\pi \tau}{12}}\prod_{n=1}^\infty \bigg(1-e^{2\pi i n \tau}\bigg).
\end{equation*}
Note that
\begin{align*}
\frac{1}{4}\sum_{(m,n)\ne (0,0)}\frac{1}{(m^2+n^2)^{s}} &= \zeta(s)\cdot L(s,\chi_{1})\\
&= \bigg(\frac{\pi/4}{s-1}+\left(\frac \pi 4\cdot \gamma_{0}+L'(1,\chi_{1})\right)+O(s-1)\bigg),
\end{align*}
since $L(1,\chi_{1}) = \pi/4$. By equating power series coefficients,
\begin{align*}
L'(1,\chi_{1}) &= \frac{\pi}{2} \bigg(\frac{\gamma_{0}}{2}-\log 2-2\log |\eta(i)|\bigg),
\end{align*}
and we conclude that
\begin{align*}
\frac{L'}{L}(1,\chi_{1}) &= \frac{4}{\pi}\cdot \frac{\pi}{2} \bigg(\frac{\gamma_{0}}{2}-\log 2-2\log |\eta(i)|\bigg)\\
&= \gamma_{0}-2\log 2-4\log |\eta(i)|.
\end{align*}
Equation (\ref{chowla}) follows from the Chowla-Selberg formula.

\end{proof}

\begin{lemma}\label{arithmetic terms}
We have
\begin{align}\label{Sl}
S_{L} = - \frac{\widehat{f}(0)}{M}\bigg(\gamma_{0}-2\log 2-4\log |\eta(i)|\bigg) + O\left(\frac{1}{M^{2}}\right),
\end{align}

and
\begin{align}\label{SA}
S_{A'} &= -\frac{2}{M}\sum_{p \equiv 3(4)}\frac{\log p}{p^2-1}\widehat{f}(0)+O\left(\frac{1}{M^{2}}\right).
\end{align}
\end{lemma}
\begin{proof}
By Cauchy's residue theorem, we may shift the contour in (\ref{L function}) to the path $\mathcal{C}_{0} \cup \mathcal{C}_{1}$, where
\begin{equation*}
\mathcal{C}_{0}:= \{\textnormal{Im}(\tau) = 0, |\textnormal{Re}(\tau)| \geq M^{\epsilon}\} \hspace{5mm} \textnormal{and} \hspace{5mm}
\mathcal{C}_{1}:= \{\textnormal{Im}(\tau) = 0, |\textnormal{Re}(\tau)|\leq M^{\epsilon}\},
\end{equation*}
for some fixed $\epsilon >0$.  By the rapid decay of $f$, for any $A>0$ we may bound
\begin{align*}
\int_{(\mathcal{C}_{0})}f(\tau)\frac{L'}{L}\left(1+\frac{2\pi i \tau}{M},\chi_{1}\right)d\tau \ll \frac{1}{M^{A\cdot \epsilon}}= O\left(\frac{1}{M^{N}}\right)
\end{align*}
where $N > 0$ may be taken to be arbitrarily large.\\
\\
Since $\mathcal{C}_{1}$ is compact, we may switch the order of integration and summation and compute the integral along $\mathcal{C}_{1}$ using the Taylor expansion of $L'/L$.  This allows us to write
\begin{align*}
\int_{(\mathcal{C}_{1})}f\left(\tau\right)\frac{L'}{L}\left(1+i\frac{2\pi \tau}{M},\chi_{1}\right) d\tau &= \frac{L'}{L}(1,\chi_{1})\int_{(\mathcal{C}_{1})}f(\tau)d\tau + O\left(\frac{1}{M}\right).
\end{align*}
By the rapid decay of $f$,
\begin{equation*}
 \int_{(\mathcal{C}_{1})}f\left(\tau\right) d\tau = \int_{\R}f\left(\tau\right) d\tau + O\left(\frac{1}{M^{100}}\right),
\end{equation*}
so that
\begin{equation*}
\int_{(\mathcal{C}_{1})}f\left(\tau\right)\frac{L'}{L}\left(1+i\frac{2\pi \tau}{M},\chi_{1}\right) d\tau =\widehat{f}(0)\frac{L'}{L}(1,\chi_{1}) + O\left(\frac{1}{M}\right).
\end{equation*}
Equation (\ref{Sl}) then follows from Lemma \ref{special value}. Equation (\ref{SA}) follows from an identical argument, upon noting that

\begin{equation*}
\frac{\partial}{\partial \alpha}A(\alpha,\alpha)\bigg|_{\alpha = 0} = -2\sum_{p \equiv 3(4)}\frac{\log p}{p^2-1}.
\end{equation*}
\end{proof}

\subsection{Computing $S_{\Gamma}$:}

\begin{lemma}\label{Sgamma lemma}
We have the asymptotic formula
\begin{align*}
S_{\Gamma} = \frac{f(0)}{2}-\frac{1}{2}\int_{\R}\widehat{f}(x) \cdot 1_{[-1,1]}(x)dx - \frac{d}{M}\cdot \widehat{f}(1)+O\left(\frac{1}{M^{2}}\right),
\end{align*}
where $d$ is as in (\ref{d constant}).
\end{lemma}
\begin{proof}
The methods employed in this section are quite similar to those found in \cite{FiorParkS}.  Set
\begin{equation*}
h(\tau) := \frac{A\left(-\frac{\pi i \tau}{M},\frac{\pi i \tau}{M}\right)}{1-\frac{2\pi i \tau}{M}}\left(\frac{\pi}{2K}\right)^{\frac{2\pi i \tau}{M}}\frac{L\left(1+\frac{2\pi i \tau}{M},\chi_{1}\right)}{L(1,\chi_{1})}\zeta\left(1-\frac{2\pi i \tau}{M}\right),
\end{equation*}

so that we may write
\begin{equation*}
S_{\Gamma}=-\frac{1}{M}\int_{(\mathcal{C}')}f(\tau)h(\tau)d\tau.
\end{equation*}
We shift the contour of integration in $S_{\Gamma}$ to the path $\mathcal{C}:= \mathcal{C}_{0} \cup \mathcal{C}_{1} \cup \mathcal{C}_{\eta}$, where
\begin{equation*}
\mathcal{C}_{0}:= \{\textnormal{Im}(\tau) = 0, \textnormal{Re}(\tau) \geq M^{\epsilon}\}, \hspace{3mm} \mathcal{C}_{1}:= \{\textnormal{Im}(\tau) = 0, \eta \leq |\textnormal{Re}(\tau)|\leq M^{\epsilon}\},
\end{equation*}
and
\begin{equation*}
\mathcal{C}_{\eta} := \{\tau=\eta e^{i \theta}, \theta \in [-\pi,0]\}.
\end{equation*}
For the part of the integral over $\mathcal{C}_{0}$, we use the rapid decay of $f$, as well as upper bounds on the components of the integrand within the critical strip, to bound

\begin{equation*}
\int_{(\mathcal{C}_{0})}f(\tau)h(\tau)d\tau \ll \frac{1}{M^{100}}.
\end{equation*}
Since the integrand is uniformly continuous on the compact set $\mathcal{C}_{1} \cup \mathcal{C}_{\eta}$, we then Taylor expand each component of the integrand, i.e. write
\begin{equation}\label{gamma factor}
\frac{1}{1-\frac{2\pi i \tau}{M}}\left(\frac{\pi}{2K}\right)^{\frac{2\pi i \tau}{M}} = e^{-2\pi i \tau }+e^{-2\pi i \tau}\frac{2\pi i \tau}{M}\bigg(\log \frac{\pi}{2}+1\bigg)+O\left(\frac{\tau^2}{M^{2}}\right),
\end{equation}

\begin{equation}\label{z factor}
\zeta\left(1-\frac{2\pi i \tau}{M}\right) = \frac{M}{2\pi \tau}i+\gamma_{0}+O\left(\frac{\tau}{M}\right),
\end{equation}

\begin{equation}\label{L factor}
\frac{L\left(1+\frac{2\pi i \tau}{M},\chi_{1}\right)}{L(1,\chi_{1})} = 1+\frac{L'}{L}(1,\chi_{1})\cdot \frac{2\pi i \tau}{M}+O\left(\frac{\tau^2}{M^2}\right),
\end{equation}
and
\begin{equation}\label{A factor}
A\left(-\frac{\pi i \tau}{M},\frac{\pi i \tau}{M}\right) = 1+2\sum_{p \equiv 3(4)}\frac{\log p}{p^{2}-1}\cdot \frac{2 \pi i \tau}{M}+O\left(\frac{\tau^2}{M^2}\right),
\end{equation}
by Lemma \ref{A derivative at zero}. Combining equations (\ref{gamma factor})$-$(\ref{A factor}), we find that
\begin{align}\label{sgamma taylor}
\begin{split}
S_{\Gamma}&= -\frac{1}{M}\int_{(\mathcal{C}_{1} \cup \mathcal{C}_{\eta})}f(\tau)h(\tau)d\tau +\left(\frac{1}{M^{100}}\right)\\
& = -\frac{1}{M}\int_{(\mathcal{C}_{1} \cup \mathcal{C}_{\eta})}f(\tau)\Bigg(e^{-2\pi i \tau }\bigg(1+\bigg(\log \frac{\pi}{2}+1\bigg)\frac{2\pi i \tau}{M}\bigg)\bigg(1+\frac{L'}{L}(1,\chi_{1})\cdot \frac{2\pi i \tau}{M}\bigg)\\
& \phantom{==}\times \bigg(\frac{M}{2\pi \tau}i+\gamma_{0}\bigg)\bigg(1+2\sum_{p\equiv 3 (4)}\frac{\log p}{p^{2}-1}\cdot \frac{2\pi i \tau}{M}\bigg)+O\left(\frac{\tau^2}{M^{2}}\right)\Bigg)d\tau+\left(\frac{1}{M^{100}}\right)\\
&= \int_{(\mathcal{C}_{1} \cup \mathcal{C}_{\eta})}f(\tau)\frac{e^{-2\pi i \tau}}{2\pi i \tau}d\tau-\frac{d}{M} \int_{(\mathcal{C}_{1} \cup \mathcal{C}_{\eta})}f(\tau)\cdot e^{-2\pi i \tau}d\tau+O\left(\frac{1}{M^{2}}\right)\\
&= \int_{(\mathcal{C}_{1} \cup \mathcal{C}_{\eta})}f(\tau)\frac{e^{-2\pi i \tau}}{2\pi i \tau}d\tau-\frac{d}{M}\cdot \widehat{f}(1)+O\left(\frac{1}{M^{2}}\right),
\end{split}
\end{align}
where $d$ is as in (\ref{d constant}), and where we note that by the holomorphy of the integrand and the rapid decay of $f$,

\begin{equation*}
\int_{(\mathcal{C}_{1} \cup \mathcal{C}_{\eta})}f(\tau)\cdot e^{-2\pi i \tau}d\tau = \widehat{f}(1)+O\left(\frac{1}{M^{100}}\right).
\end{equation*}
Since $\mathcal{C}:= \mathcal{C}_{0} \cup \mathcal{C}_{1} \cup \mathcal{C}_{\eta}$, we may moreover write
\begin{align}\label{J1 and J2}
\begin{split}
\int_{(\mathcal{C}_{1} \cup \mathcal{C}_{\eta})}f(\tau)\frac{e^{-2\pi i \tau}}{2\pi i \tau}d\tau &= \frac{1}{2\pi i}\int_{(\mathcal{C})}\frac{f(\tau)}{\tau}e^{-2\pi i \tau}d\tau+O\left(\frac{1}{M^{100}}\right)\\
&= J_{1}+J_{2}+O\left(\frac{1}{M^{100}}\right),
\end{split}
\end{align}
where
\begin{equation*}
J_{1}:= \frac{1}{2\pi i}\int_{(\mathcal{C})}\cos(2\pi \tau)\frac{f(\tau)}{\tau}d\tau,
\end{equation*}
and
\begin{equation*}
J_{2}:= -\int_{(\mathcal{C})}\frac{\sin(2\pi \tau)}{2\pi \tau}f(\tau)d\tau,
\end{equation*}
both of which are well-defined by the rapid decay of $f$.\\
\\
\textbf{Computing $J_{1}$:}
Since the integrand in $J_{1}$ is odd, integration over the path $\mathcal{C}_{0} \cup \mathcal{C}_{1}$ is zero.  Hence
\begin{align*}
J_{1}&= \frac{1}{2\pi i}\int_{\mathcal{C}_{\eta}}\cos(2\pi \tau)\frac{f(\tau)}{\tau}d\tau= \frac{1}{2\pi}\int_{-\pi}^{0}\cos(2\pi \eta e^{i \theta})f\left(\eta e^{i \theta}\right)d\theta,
\end{align*}
where we apply the change of variables $\tau = \eta e^{i \theta}$.  Since
\begin{equation*}
\cos(2\pi \eta e^{i \theta}) = 1+O\left(\eta^{2}\right),
\end{equation*}
we find that in the limit as $\eta \rightarrow 0$,
\begin{align} \label{J1}
J_{1} = \lim_{\eta \rightarrow 0}\frac{1}{2\pi}  \int_{-\pi}^{0}f\left(\eta e^{i \theta}\right)d\theta &=
\frac{1}{2\pi}\cdot \pi f(0)= \frac{f(0)}{2}.
\end{align}

\textbf{Computing $J_{2}$:}
Note that since $\sin(2\pi \tau)/2\pi \tau$ is entire except for a removable singularity at $\tau=0$, we may shift the line of integration to obtain
\begin{equation} \label{J2}
J_{2} = -\int_{\R}\frac{\sin(2\pi \tau)}{2\pi \tau}f(\tau)d\tau = -\frac{1}{2}\int_{-1}^{1}\widehat{f}(\tau)d\tau,
\end{equation}
as in (\ref{small support symplectic}). This coincides with the second term in the Katz-Sarnak prediction.  Since all the error terms are independent of $\eta$, Lemma \ref{Sgamma lemma} then follows from (\ref{sgamma taylor}), (\ref{J1 and J2}), (\ref{J1}), and (\ref{J2}).
\end{proof}

Conjecture \ref{ratios prediction} now follows from Conjecture \ref{1-c line} and Lemmas \ref{computing Wf}, \ref{Szeta}, \ref{arithmetic terms}, and \ref{Sgamma lemma}.  In particular, when supp$(\widehat{f}) \subseteq (-1,1)$, the conjecture implies that,
\begin{align*}
D_{1}(\mathcal{F}(K);f)&= \widehat{f}(0)-\frac {f(0)}{2}+c\cdot \frac{\widehat{f}(0)}{M}+O\left(\frac{1}{M^{2}}\right),
\end{align*}
in agreement with Theorem \ref{oneleveldensity}.
\begin{Remark}\label{remark lower order accuracy}
 \textnormal{Upon applying the convolution theorem, we note that when supp$(\widehat{f}) \subseteq (-1,1)$,
\[\int_{\R}f(\tau)\tau^{n} e^{-2\pi i \tau}d\tau = 0\]
for any $n \in \N$.  Hence, upon computing additional terms in the Taylor series expansion of $h(\tau)$, it may be shown that in fact supp$(\widehat{f}) \subseteq (-1,1) \Rightarrow S_{\Gamma} \ll M^{-n}$ for any fixed $n > 0.$}
\end{Remark}
\section{Proof of Theorem \ref{oneleveldensity} and Corollary \ref{nonvanish}}

\subsection{Explicit Formula}
As in (\ref{one level}), we note that

\begin{equation*}
D_{1}(\mathcal{F}(K);f) = \frac{1}{2\pi i}\int_{(c')}\Bigg(2\bigg \langle \frac{L'_{k}}{L_{k}}\left(\frac 1 2+r\right)\bigg \rangle_{K}-\bigg \langle \frac{X'_{k}}{X_{k}}\left(\frac 1 2+r\right)\bigg \rangle_{K}\Bigg)f\left(\frac{i M r}{\pi}\right)dr,
\end{equation*}
where now, unconditionally, we choose $c'\geq 1/2$.  By (\ref{Lk log derivative}) and an identical argument to that found in Lemma \ref{zeta and L as sum},

\begin{align*}
\frac{1}{\pi i}\int_{(c')}\frac{L_{k}'}{L_{k}}\bigg(\frac{1}{2}+r\bigg)f\left(\frac{i M r}{\pi}\right)dr&= -\frac{1}{M}\sum_{n=1}^{\infty}\frac{c_{k}(n)}{\sqrt{n}}\widehat{f}\left(\frac{\log n}{2M}\right).
\end{align*} 

It follows that the averaged scaled one-level density is given by

\begin{align}\label{GRH result}
D_{1}(\mathcal{F}(K);f)&= W_{f}+S_{\text{split}}+S_{\text{inert}}+S_{\text{ram}},
\end{align}
where

\begin{equation*}
S_{\text{inert}} := -\frac{2}{M}\sum_{\substack{p \textnormal{ prime }\\p\equiv 3(4)\\ l \in 2\N}}\frac{\Lambda(p^{l})}{p^{l/2}}\widehat{f}\left(\frac{l\textnormal{ log }p}{2M}\right),
\end{equation*}

\begin{equation*}
S_{\text{ram}} := - \frac{1}{M}\frac{1}{K}\sum_{k=1}^{K}\sum_{l=1}^{\infty}\frac{(-1)^{lk}\Lambda(2^{l})}{2^{l/2}}\widehat{f}\left(\frac{l \textnormal{ log }2}{2M}\right),
\end{equation*} 
and
\begin{equation*}
\begin{split}
S_{\text{split}} := -\frac{1}{M}\sum_{\substack{p \textnormal{ prime }\\ p\equiv 1(4)\\ l \in \mathbb{N}}}\frac{\Lambda(p^{l})}{p^{l/2}}\widehat{f}\left(\frac{l\cdot \textnormal{log }p}{2M}\right)\bigg(\langle \Xi_{k}(\p_{1}^{l})\rangle_{K}+\langle \Xi_{k}(\p_{2}^{l})\rangle_{K}\bigg).
\end{split}
\end{equation*}
We now proceed to compute $S_{\text{inert}}$, $S_{\text{ram}}$, and $S_{\text{split}}$.

\subsection{Ramified and Inert Primes}\label{ramified}
We compute
\begin{align*}
&S_{\text{ram}} = - \frac{1}{M}\frac{1}{K}\sum_{k=1}^{K}\sum_{l=1}^{\infty}\frac{(-1)^{lk}\Lambda(2^{l})}{2^{l/2}}\widehat{f}\left(\frac{l\textnormal{ log }2}{2M}\right)\\
&=- \frac{1}{M}\frac{1}{K}\Bigg(\sum_{k=1}^{K}\sum_{\substack{l=2\\ \textnormal{even}}}^{\infty}\frac{\log 2}{2^{l/2}}\widehat{f}\left(\frac{l \textnormal{ log }2}{2M}\right)+\sum_{k=1}^{K}(-1)^{k}\sum_{\substack{l=1\\ \textnormal{odd}}}^{\infty}\frac{\log 2}{2^{l/2}}\widehat{f}\left(\frac{l \textnormal{ log }2}{2M}\right)\Bigg)\\
&=- \frac{\log 2}{M}\sum_{n=1}^{\infty}\frac{1}{2^{n}}\widehat{f}\left(\frac{n \textnormal{ log }2}{M}\right)+O\left(\frac{1}{K}\right).
\end{align*}
Upon noting that
\begin{equation*}
\chi_{1}(n)-1=
\left\{
\begin{array}{l l}
- 2 & \text{ if } n \equiv 3(4) \\
0 & \text{ if } n \equiv 1(4) \\
- 1 & \text{ otherwise},\\
\end{array} \right.
\end{equation*}
it follows from Lemma \ref{zeta and L as sum} that
\begin{align*}
S_{\zeta}+S_{L}&=\frac{1}{M}\sum_{n=1}^{\infty}\frac{\Lambda(n)(\chi_{1}(n)-1)}{n}\widehat{f}\left(\frac{\log n}{M}\right)\\
&=-\frac{2}{M}\sum_{\substack{n \in \N \\ n \equiv 3(4)}}\frac{\Lambda(n)}{n}\widehat{f}\left(\frac{\log n}{M}\right)- \frac{1}{M}\sum_{n=1}^{\infty}\frac{\log 2}{2^{n}}\widehat{f}\left(\frac{n\textnormal{ log }2}{M}\right)\\
&=-\frac{2}{M}\sum_{p \equiv 3(4)}^{\infty}\sum_{n=1}^{\infty}\frac{\log p}{p^{2n+1}}\widehat{f}\left(\frac{(2n+1) \log p}{M}\right)+S_{\text{ram}}+O\left(\frac{1}{K}\right)\\
&=-\frac{2}{M}\sum_{p\equiv 3(4)}\sum_{n=1}^{\infty}\frac{\log p}{p^{n}}\widehat{f}\left(\frac{n \textnormal{ log }p}{M}\right)\\
&\hspace{4mm}+\frac{2}{M}\sum_{p\equiv 3(4)}\sum_{n=1}^{\infty}\frac{\log p}{p^{2n}}\widehat{f}\left(\frac{2n \textnormal{ log }p}{M}\right)+S_{\text{ram}}+O\left(\frac{1}{K}\right)\\
&= S_{\text{inert}} - S_{A'}+S_{\text{ram}}+O\left(\frac{1}{K}\right),
\end{align*}
yielding (\ref{GRH Ratios Relation}).
\subsection{Split Primes}
By (\ref{GRH Ratios Relation}) and (\ref{GRH result}) we obtain the following corollary:
\begin{corollary}\label{character sum}
If Conjecture \ref{1-c line} is true, then
\begin{align*}
S_{\text{split}} &= S_{\Gamma} +O\left(e^{-\frac{M}{2}}\right)\\
&= \frac{f(0)}{2}-\frac{1}{2}\int_{\R}\widehat{f}(x) \cdot 1_{[-1,1]}(x)dx - \frac{c}{M}\cdot \widehat{f}(1)+O\left(\frac{1}{M^{2}}\right).
\end{align*}
\end{corollary}

\subsubsection{Proof of Theorem \ref{oneleveldensity}}
We now compute $S_{\text{split}}$ when supp$(\widehat{f}) \subseteq  [-\alpha,\alpha]$ for some $\alpha < 1$.  First, bound

\begin{align}\label{split bound}
S_{\text{split}} &\ll \frac{1}{M}\sum_{\substack{\A \subset \Z[i]\\ \theta_{\A}\neq 0}}\frac{\Lambda(N(\A))}{N(\A)^{1/2}}\bigg |\widehat{f}\left(\frac{\log N(\A)}{2M}\right)\bigg | \cdot |\langle \Xi_{k}(\A)\rangle_{K}|.
\end{align}

Since $\theta_{\A} \notin  \pi \Z$, we have that
\begin{align}\label{angle average}
\langle \Xi_{k}(\A)\rangle_{K} &= \frac{1}{K}\sum_{k =1}^{K}e^{4ki \theta_{\A}} = \frac{1}{K}\cdot \frac{\Xi_{1}(\A)-\Xi_{K}(\A)}{1-\Xi_{1}(\A)}\ll \frac{1}{K \theta_{\A}},
\end{align}

where we note that 
\begin{equation*}
|1-\Xi_{1}(\A)| = 2|\sin(2 \theta_{\A})|\gg \theta_{\A}.
\end{equation*}

Moreover, by Lemma 2.1 in \cite{RudWax}, we have that
\begin{equation}\label{angle bound}
\theta_{\A} \gg \frac{1}{\sqrt{N(\A)}},
\end{equation}
i.e. the smallest possible angle attached to an ideal of norm $x$ is $\gg x^{-1/2}$.\\
\\
To each ideal $\A \subset \mathbb{Z}[i]$, we associate an annulus sector $S_{\A}$ of fixed area, enclosing the lattice point of radius $r_{\A} =\sqrt{N(\A)}$ and angle $\theta_{\A}$.  We moreover choose each $S_{\A}$ to be sufficiently small, so that $S_{\A} \cap S_{\A'}=0$ for any $\A \neq \A'$.  Then

\begin{equation}\label{local bound}
\frac{\Lambda(N(\A))}{\sqrt{N(\A)}}\bigg |\widehat{f}\left(\frac{\log N(\A)}{2M}\right)\bigg|\frac{1}{\theta_{\A}} \ll \iint_{S_{\A}} \log r\bigg |\widehat{f}\left(\frac{\log r_{\A}'}{M}\right)\bigg | \frac{d\theta}{\theta} dr,
\end{equation}
where $\iint_{S_{\A}} d\theta r dr $ denotes the integral over $S_{\A}$, and where each $r_{\A}' \in S_{\A}$ is chosen such that
\[\bigg |\widehat{f}\left(\frac{\log r'_{\A}}{M}\right)\bigg | \geq \bigg |\widehat{f}\left(\frac{\log r}{M}\right)\bigg |\]
 for all $re^{i\theta} \in S_{\A}$.  By (\ref{split bound}), (\ref{angle average}), (\ref{angle bound}), and (\ref{local bound}), and upon recalling that supp$(\widehat{f}) \subseteq  [-\alpha,\alpha]$, we find that

\begin{align*}
S_{\text{split}} &\ll \frac{1}{K}\frac{1}{M}\sum_{\substack{\A \subset \Z[i]\\ \theta_{\A}\neq 0}} \frac{\Lambda(N(\A))}{\sqrt{N(\A)}}\bigg |\widehat{f}\left(\frac{\log N(\A)}{2M}\right)\bigg|\frac{1}{\theta_{\A}}\\
&\ll \frac{1}{K}\frac{1}{M}\sum_{\substack{\A \subset \Z[i]\\ \theta_{\A}\neq 0}} \iint_{S_{\A}} \log r \bigg |\widehat{f}\left(\frac{ \log r'_{\A}}{M}\right)\bigg | \frac{d\theta}{\theta}dr\\
&\ll \frac{1}{K}\frac{1}{M}\int_{1}^{K^{\alpha}}\log r \bigg(\int_{r^{-1}}^{\frac{\pi}{2}}\frac{d \theta}{\theta}\bigg) dr\\
&\ll \frac{1}{K}\frac{1}{M}\int_{1}^{K^{\alpha}}(\log r)^{2}dr\\
& \ll K^{\alpha-1}\cdot \log K.
\end{align*}

In particular, we conclude that when $\alpha< 1$,

\begin{equation}\label{split}
S_{\text{split}} = O_{\alpha}\left(K^{-\delta}\right)
\end{equation}
for some $\delta >0$. Theorem \ref{oneleveldensity} now follows upon collecting the results in (\ref{GRH Ratios Relation}), Lemma \ref{computing Wf}, Lemma \ref{Szeta}, Lemma \ref{arithmetic terms}, (\ref{GRH result}), and (\ref{split}).

\begin{Remark}\label{remark unconditional match}
\textnormal{From (\ref{split}) and Remark \ref{remark lower order accuracy} it follows that if supp$(\widehat{f}) \subset [-\alpha,\alpha]$ for some $\alpha< 1$, then $D_{1}(\mathcal{F}(K);f)$ agrees with the prediction of the Ratios Conjecture down to an accuracy of size $O(M^{-n})$, for any $n \geq 2$.}
\end{Remark}

\subsubsection{Proof of Corollary \ref{nonvanish}}
\begin{proof}
Let
\begin{equation*}
p_{m}(K) = \frac{1}{K}\#\{L_{k}(s) \in \mathcal{F}(K):\textnormal{ord}_{s=\frac 1 2}L_{k}(1/2)=m\},
\end{equation*}
and choose an admissible test function $f(x) \geq 0$ such that $f(0) =1$.  Then under the assumption of GRH,
\begin{equation*}
\sum_{m=1}^{\infty}mp_{m}(K)\leq D_{1}(\mathcal{F}(K);f).
\end{equation*}
Moreover, by the symmetry of the functional equation, we find that the order of vanishing at the central point must be even, 
i.e. $p_{2m+1}(K)=0$.  If supp$(\widehat{f}) \subset (-1,1)$, we then obtain 
\begin{align*}
p_{0}(K) &= 1-\sum_{m=1}^{\infty}p_{m}(K)\geq 1-\frac{1}{2}D_{1}(\mathcal{F}(K);f)= \frac{5}{4}-\frac{\widehat{f}(0)}{2}+o(1),
\end{align*}
by Theorem \ref{oneleveldensity}. Corollary \ref{nonvanish} then follows upon choosing the Fourier pair
\begin{equation*}
f(x) = \left(\frac{\sin (\nu \pi x)}{\nu \pi x}\right)^{2}, \hspace{5mm} \widehat{f}(y) = \left\{
\begin{array}{l l}
\frac{1}{\nu}\left(1-\frac{|y|}{\nu}\right) & \textnormal{ if } |y|< \nu. \\
0 & \text{ otherwise }
\end{array} \right.
\end{equation*}

and then taking the limit as $a \rightarrow 1^{-}$.
\end{proof}

\end{document}